\documentclass{amsart}
\usepackage{amsmath}
\usepackage{amssymb}
\usepackage{amsthm}
\usepackage{enumerate}
\usepackage[hyphens]{url}
\usepackage{hyperref}
\usepackage{tikz}
\usepackage{tikz-cd}
\hypersetup{
    colorlinks=true,
    linkcolor=blue,
}

\usepackage[
    backend=biber,
    style=alphabetic,
    useprefix=true,
    doi=false,
    maxnames=99
]{biblatex}
\newtheorem{theorem}{Theorem}[section]
\newtheorem{proposition}[theorem]{Proposition}
\newtheorem{lemma}[theorem]{Lemma}
\newtheorem{corollary}[theorem]{Corollary}

\theoremstyle{definition}
\newtheorem{definition}[theorem]{Definition}
\newtheorem{example}[theorem]{Example}
\newtheorem{construction}[theorem]{Construction}
\newtheorem{remark}[theorem]{Remark}

\newtheorem{situation}[theorem]{Situation}

\renewcommand{\th}[1]{#1\textsuperscript{th}}

\newcommand{\<}{\left<}
\renewcommand{\>}{\right>}

\newcommand{\N}{\mathbb{N}}
\newcommand{\Z}{\mathbb{Z}}
\newcommand{\Q}{\mathbb{Q}}
\newcommand{\C}{\mathbb{C}}

\newcommand{\G}{\mathbb{G}}
\newcommand{\A}{\mathbb{A}}
\renewcommand{\P}{\mathbb{P}}

\renewcommand{\O}{\mathcal{O}}
\newcommand{\I}{\mathcal{I}}
\newcommand{\M}{\mathcal{M}}
\newcommand{\R}{\mathcal{R}}
\renewcommand{\AA}{\mathcal{A}}
\newcommand{\CC}{\mathcal{C}}
\newcommand{\DD}{\mathcal{D}}
\newcommand{\EE}{\mathcal{E}}
\newcommand{\FF}{\mathcal{F}}
\newcommand{\GG}{\mathcal{G}}
\newcommand{\HH}{\mathcal{H}}
\newcommand{\NN}{\mathcal{N}}

\newcommand{\X}{\mathfrak{X}}
\newcommand{\Y}{\mathfrak{Y}}
\newcommand{\ii}{\mathfrak{i}}

\newcommand{\D}{\mathsf{D}}
\newcommand{\DF}{\mathsf{DF}}
\newcommand{\Ani}{\mathsf{Ani}}
\newcommand{\Sp}{\mathsf{Sp}}
\newcommand{\Cat}{\mathsf{Cat}}
\newcommand{\Fun}{\mathsf{Fun}}
\renewcommand{\Pr}{\mathsf{Pr}}
\newcommand{\PS}{\mathsf{P}_\Sigma}

\newcommand{\Ring}{\mathsf{Ring}}
\newcommand{\Mod}{\mathsf{Mod}}
\newcommand{\CAlg}{\mathsf{CAlg}}
\newcommand{\Gr}{\mathsf{Gr}}
\newcommand{\Fil}{\mathsf{Fil}}

\newcommand{\Bun}{\mathsf{VB}}
\newcommand{\Pair}{\mathsf{Pair}}
\renewcommand{\Gauge}{\mathsf{Gauge}}
\newcommand{\FGauge}{{F\text{-}\mathsf{Gauge}}}
\newcommand{\Fin}{\mathsf{Fin}}

\newcommand{\E}{\mathrm{E}}
\renewcommand{\H}{\mathrm{H}}

\newcommand{\RG}{\mathrm{R}\Gamma}
\renewcommand{\L}{\mathrm{L}}
\newcommand{\LL}{\mathrm{L}\Lambda}
\newcommand{\LO}{\mathrm{L}\Omega}
\newcommand{\B}{\mathrm{B}}
\renewcommand{\d}{\mathrm{d}}
\newcommand{\dlog}{\mathrm{dlog}}
\newcommand{\op}{\mathrm{op}}
\newcommand{\GL}{\mathrm{GL}}

\newcommand{\aff}{\mathrm{Aff}}
\newcommand{\syn}{\mathrm{syn}}
\newcommand{\et}{\mathrm{\acute{e}t}}
\newcommand{\proet}{\mathrm{pro\acute{e}t}}

\newcommand{\bl}{\mathrm{Bl}}
\newcommand{\Th}{\mathrm{Th}}
\newcommand{\dr}{\mathrm{dR}}
\newcommand{\HT}{\mathrm{HT}}
\newcommand{\hod}{\mathrm{Hod}}

\newcommand{\wcart}{\mathrm{WCart}}
\newcommand{\OD}{\Omega^{\not{D}}}
\newcommand{\st}{\mathrm{st}}
\newcommand{\arc}{\mathrm{arc}}
\newcommand{\cyc}{\mathrm{cyc}}
\newcommand{\perf}{\mathrm{perf}}
\newcommand{\cc}{\mathrm{c}}
\newcommand{\cl}{\mathrm{Cl}}
\newcommand{\THH}{\mathrm{THH}}
\newcommand{\TC}{\mathrm{TC}}

\newcommand{\sel}{\mathrm{Sel}}
\newcommand{\h}{\mathrm{h}}
\newcommand{\reg}{\mathrm{reg}}
\newcommand{\pair}{\mathrm{Pair}}
\newcommand{\closed}{\mathrm{Closed}}
\newcommand{\open}{\mathrm{Open}}
\newcommand{\DNC}{\mathrm{DNC}}

\DeclareMathOperator{\id}{id}
\DeclareMathOperator{\Hom}{Hom}
\DeclareMathOperator{\RHom}{RHom}
\DeclareMathOperator{\inhom}{\underline{Hom}}
\DeclareMathOperator*{\colim}{colim}
\DeclareMathOperator{\spec}{Spec}
\DeclareMathOperator{\spf}{Spf}
\DeclareMathOperator{\spa}{Spa}
\DeclareMathOperator{\spd}{Spd}
\DeclareMathOperator{\proj}{Proj}
\DeclareMathOperator{\fib}{fib}
\DeclareMathOperator{\cofib}{cofib}
\DeclareMathOperator{\eq}{eq}
\DeclareMathOperator{\fil}{Fil}
\DeclareMathOperator{\gr}{gr}
\DeclareMathOperator{\const}{const}
\DeclareMathOperator{\triv}{triv}
\DeclareMathOperator{\rk}{rk}
\DeclareMathOperator{\lsym}{LSym}
\DeclareMathOperator{\tor}{Tor}

\usepackage{relsize}
\usepackage[bbgreekl]{mathbbol}
\usepackage{amsfonts}
\DeclareSymbolFontAlphabet{\mathbb}{AMSb}
\DeclareSymbolFontAlphabet{\mathbbl}{bbold}
\newcommand{\Prism}{{\mathlarger{\mathbbl{\Delta}}}}

\begin{document}

\title{Syntomic cycle classes and prismatic Poincar\'e duality}
\author{Longke Tang}

\begin{abstract}
We introduce $F$-gauges over a prism, construct syntomic cycle classes, and prove the prismatic Poincar\'e duality for proper smooth schemes. 
\end{abstract}

\maketitle

\setcounter{tocdepth}{1}

\tableofcontents

\section{Introduction}

\subsection{Goal}
Poincar\'e duality is one of the most classical theorems in algebraic topology, saying that if $X$ is a compact oriented manifold of dimension $d$, then its singular cohomologies with coefficient in a field $k$ satisfy
$$\H^i(X,k)^\vee\cong\H^{d-i}(X,k),$$
where $-^\vee$ denotes dual $k$-vector spaces. 

If $X$ is now a projective complex manifold of dimension $d$, then the Hodge decomposition theorem states that
$$\H^i(X,\Z)\otimes_\Z\C\cong\bigoplus_{p+q=i}\H^q(X,\Omega^p_{X/\C}),$$
where $\Omega^p_{X/\C}$ is the sheaf of holomorphic $p$-forms on $X$, and $\H^q(X,-)$ means the sheaf cohomology. From the proof of the above isomorphism, it is straightforward to see that the Poincar\'e duality on the left and the Serre duality on the right intertwine. Therefore, in terms of Hodge structures, one can write the duality as
\begin{equation}\label{complexdual}
    \RG(X,\Z)^\vee\cong\RG(X,\Z)(d)[2d],
\end{equation}
where $-^\vee$ denotes the dual functor $\RHom(-,\Z)$ in the derived category of Hodge structures. See for example \cite[Example 3.1.5, Theorem 3.1.17]{cht}. 

An important consequence of the Poincar\'e duality is the cycle class, which associates a cohomology class $\cl_{Y/X}\in\H^r(X,\Z)$ to a compact oriented submanifold $Y$ of codimension $r$ in a compact oriented manifold $X$. 

If $X$ is again a projective complex manifold and $Y$ is its closed submanifold of complex codimension $r$, then the above cycle class, after tensoring with $\C$, actually lies in the direct summand $\H^r(X,\Omega^r_{X/\C})$ in the Hodge decomposition. See for example \cite[Lemma 3.1.27]{cht}. In terms of Hodge structures, this amounts to saying that the cycle class defines a map of Hodge structures
\begin{equation}\label{complexcycle}
    \cl_{Y/X}\colon\Z\to\H^{2r}(X,\Z)(r).
\end{equation}

This paper aims to tell the above story for integral $p$-adic Hodge theory. 

\subsection{Results}
In their remarkable paper \cite{bs19}, Bhatt and Scholze defined the prismatic cohomology for schemes over a $p$-adic base, which specializes to various $p$-adic cohomologies known before, including \'etale, de Rham, and crystalline, also generalizing their previous works \cite{bms1} and \cite{bms2} joint with Morrow. 

As its name suggests, prismatic cohomology is based on prisms, which are roughly triples $(A,I,\varphi)$, where $A$ is a ring, $I$ is an invertible ideal, $\varphi\colon A\to A$ is a Frobenius lift, $A$ is $(p,I)$-complete, and finally $p\in(I,\varphi(I))$. See \cite[Definition 3.2]{bs19} for a precise definition. Given such a structure, prismatic cohomology associates to every $A/I$-scheme $X$ a complex $\RG_\Prism(X/A)\in\D(A)$, along with a $\varphi$-semilinear Frobenius $F\colon\RG_\Prism(X/A)\to\RG_\Prism(X/A)$ and a so-called Nygaard filtration $\fil^\bullet\varphi^*\RG_\Prism(X/A)$, from which one can recover the \'etale, de Rham, and crystalline cohomologies of $X$; see \cite[Theorem 1.8, Theorem 1.16]{bs19}. 

Given the above construction, the reader might have noticed that, in the $p$-adic setting, prismatic cohomology should play the role of singular cohomology with a Hodge structure, and the outputs of prismatic cohomology should play the role of Hodge structures. This is indeed the case. The structures in these outputs are summarized in \S\ref{secfg} as follows: 

\begin{definition}[Definitions \ref{deffg}, \ref{hts}, \ref{drs}, \ref{ets}; Examples \ref{bk}, \ref{pc}]
An \emph{$F$-gauge} over a prism $(A,I)$ is a triple $(M,\fil^\bullet\varphi^*M,F)$ where $\fil^\bullet\varphi^*M$ is a filtration on $\varphi^*M$ and $F\colon\fil^\bullet\varphi^*M\to I^\bullet M$ exhibits $I^\bullet M$ as a filtered base change of $\fil^\bullet\varphi^*M$ along $A\to A[1/I]$. The category of $F$-gauges over $(A,I)$ is denoted $\FGauge(A,I)$. 

The outputs of prismatic cohomology are $F$-gauges, and there are various specialization functors from $\FGauge(A,I)$, corresponding to Hodge--Tate, de Rham, and \'etale specializations of prismatic cohomology. In addition, there are \emph{Breuil--Kisin twists} $A\{n\}$ of $F$-gauges for all $n\in\Z$ that are analogous to Tate twists $\Z(n)$ of complex Hodge structures. 
\end{definition}

With $F$-gauges in hand, one can now easily write down the $p$-adic analog of the isomorphism (\ref{complexdual}), the prismatic Poincar\'e duality:

\begin{theorem}[Theorem \ref{prismdual}]
For a proper smooth $A/I$-scheme $X$ of dimension $d$, there is a canonical isomorphism of $F$-gauges
$$\RG_\Prism(X/A)^\vee\cong\RG_\Prism(X/A)\{d\}[2d].$$
\end{theorem}

To prove the above theorem is not as easy as to state it. In the case of \'etale cohomology, with the help of its 6-functor formalism, one essentially localizes, factorizes, and finally treats $\A^1$. While it might be possible, we choose not to develop a prismatic 6-functor formalism here. In the case of de Rham cohomology, \cite[\texttt{0FW3}]{stacks} tried hard and managed to construct a trace map $\RG_\dr(X/k)\to k$ that induces a perfect pairing after the cup product; we are not sure whether this method will work here, either. Instead, we follow and improve the method of the recent lecture notes \cite{dr}, which first constructed the de Rham cycle class, then applied it to the diagonal, obtaining a copairing
$$k\to\RG_\dr(X/k)\otimes\RG_\dr(X/k)[2d],$$
and finally proved its perfectness by reduction to the Hodge case. Note that the logical order here is reverse to that in classical algebraic topology, where one constructs the cycle class using the Poincar\'e duality. 

As suggested by the map (\ref{complexcycle}) in the complex case, the prismatic cycle class should actually be an $F$-gauge map 
$$\cl_{Y/X}^\Prism\colon A\to\RG_\Prism(X/A)\{r\}[2r],$$
instead of merely an element of $\H^{2r}_\Prism(X/A)$. This brings us to consider the following:

\begin{definition}[Proposition \ref{glsect}, cf.\ {\cite[\S 4]{ktannounce}}]
For an $A/I$-scheme $X$, its \emph{\th{$n$} syntomic cohomology relative to $A$}, denoted $\Z_p(n)(X/A)$, is the complex
$$\RHom_{\FGauge(A,I)}(A,\RG_\Prism(X/A)\{n\})\in\D(\Z_p).$$
\end{definition}

While the relative syntomic cohomology seems hard to control, its absolute counterpart has long been considered by $K$-theorists as the $p$-adic motivic cohomology, cf.\ \cite{nizmot}, and has recently been redefined using the absolute prismatic cohomology by \cite{bms2} and \cite{apc}. Since the absolute syntomic cohomology naturally maps to the relative one, construction of the relative syntomic cycle class boils down to that of the absolute one, which is another main result of this paper:

\begin{theorem}[Definitions \ref{defcyc}, \ref{cyclemap}; Remarks \ref{unicyc}, \ref{mult}]
There is a unique functorial way to associate a syntomic cohomology class $\cl_{Y/X}^\syn\in\H^{2r}_\syn(X,\Z_p(r))$ to every regular immersion $Y\to X$ of codimension $r$, satifying:
\begin{description}
    \item[Multiplicativity] For regular immersions $Y\to X$ and $Y'\to X'$, we have $\cl_{Y/X}^\syn\boxtimes\cl_{Y'/X'}^\syn=\cl_{(Y\times Y')/(X\times X')}^\syn$.
    \item[Normalization for divisors] In codimension one, $\cl_{Y/X}^\syn=c_1(\O(Y))$ is the first Chern class defined in \cite[\S 7.5, \S 8.4]{apc}. 
\end{description}
Moreover, the syntomic cycle class can be upgraded to a syntomic Gysin map
\begin{equation}\label{cycmapintro}
    \RG_\syn(Y,\Z_p(n))\to\RG_\syn(X,\Z_p(n+r))[2r].
\end{equation}
\end{theorem}

We also relate the syntomic Gysin map above to the pushforward map in $K$-theory via the motivic filtration. In the following theorem, let $K$ denote the $p$-completed \'etale $K$-theory, which is motivic filtered by $\Z_p(n)$ by \cite[\S 7.4]{bms2}. 

\begin{theorem}[Theorem \ref{cyctc}]
There is a unique way to enhance the pushforward $K(Y)\to K(X)$ to a motivic filtered map $\fil^\bullet K(Y)\to\fil^{\bullet+r}K(X)$ functorially for all regular immersions $Y\to X$ of codimension $r$ of $p$-formal stacks. Moreover, its associated graded map coincides with the map (\ref{cycmapintro}). 
\end{theorem}

\begin{remark}
    The functoriality in the above two theorems is with respect to only pullbacks of regular immersions rather than compositions of them. We are not going to pursue the latter functoriality in this work, but we hope to address it in a future project, where a more general Gysin map construction based on the $\P^1$-motivic homotopy theory developed in \cite{ahi23p1} will be given. 
\end{remark}

\subsection{Conventions}
Fix a prime $p$ throughout. By ``rings'' we mean commutative rings. Everything is animated or derived unless otherwise stated, that is: 
\begin{itemize}
    \item By ``categories'' we mean $\infty$-categories, unless we say ``$1$-categories''; 
    \item By ``modules'' we mean complexes, unless we say ``classical modules''; 
    \item By ``rings'' we mean animated rings, unless we say ``classical rings''; 
    \item \ldots
\end{itemize}
But by ``sets'' we still mean $0$-groupoids, because we have ``animas'' for $\infty$-groupoids. 

\subsection{Acknowledgements}
First, I heartily thank my advisor Bhargav Bhatt for encouraging me to work on prismatic Poincar\'e duality and suggesting that I should imitate the lecture notes \cite{dr}. Next, I would like to thank Dustin Clausen for these lecture notes that give a beautiful proof of the duality in the de Rham case. Moreover, I thank Toni Annala, Johannes Anschütz, Benjamin Antieau, Arthur-César Le Bras, Elden Elmanto, Ofer Gabber, Haoyang Guo, Lars Hesselholt, Marc Hoyois, Teruhisa Koshikawa, Dmitry Kubrak, Yixiao Li, Shiji Lyu, Lucas Mann, Akhil Mathew, Emanuel Reinecke, Yuchen Wu, Bogdan Zavyalov, and Daming Zhou for helpful discussions. Finally, I would also like to thank the referee for many helpful comments improving the exposition and precision of the paper, and the Hausdorff Research Institute for Mathematics for their hospitality and support during the trimester program on the arithmetic of the Langlands program. 

\section{\texorpdfstring{$F$}{F}-gauges over a prism}\label{secfg}

We start by defining the notions of gauges and $F$-gauges over a prism, which are meant to capture the bunch of structures in the relative prismatic cohomology. Thus in this section we fix an animated prism $(A,I)$ as defined in \cite[Definition 2.4]{pst}, and everything we are going to talk about will be $(p,I)$-complete whenever it is over $A$, unless otherwise stated. For example, $\D(A)$ will denote the $(p,I)$-complete derived category of $A$. Readers not familiar with animated prisms can assume that $(A,I)$ is a bounded classical prism. Let $\bar{A}=A/I$. 

\subsection{Reminders on filtered objects}
In this subsection we recall the definition of filtered objects and some basic constructions on them, roughly following \cite[\S 1.2.2]{ha} and \cite[\S 3.1, \S 3.2]{rak20}. Let $\CC$ be a stable $\infty$-category. 

\begin{definition}[Filtered objects]\label{dfft}
    A \emph{filtered object} $\fil^\bullet M$ of $\CC$ is a functor $\Z^\op\to\CC$, where $\Z=(\Z,\le)$ is the poset viewed as a category. We denote its value on $n\in\Z$ by $\fil^nM$. The \emph{\th{$n$} graded piece} of $\fil^\bullet M$ is defined as $\fil^nM/\fil^{n+1}M$ and is denoted $\gr^nM$. If $\fil^\bullet M$ has all its graded pieces vanishing, namely it takes all arrows in $\Z^\op$ to isomorphisms, then we call it a \emph{constant} filtered object. 
    
    The \emph{category of filtered objects} of $\CC$ is defined as the functor category $\Fun(\Z^\op,\CC)$ and is denoted $\Fil(\CC)$. All the graded pieces assemble to a functor $\gr\colon\Fil(\CC)\to\CC^\Z$, called taking \emph{associated graded objects}. Taking constant objects is also a functor $\const\colon\CC\to\Fil(\CC)$. 
\end{definition}

\begin{remark}[Increasing filtrations]
    What Definition \ref{dfft} defines are the so-called \emph{decreasing filtrations}, in that if we do the same construction for an abelian category and require all the transition maps to be injective, then $\fil^nM$ gets smaller when $n$ gets larger. One can similarly define \emph{increasing filtrations} as functors $\Z\to\CC$, where the same theory applies with indices reversed. We will sometimes use increasing filtrations below, and these will be denoted by $\fil_\bullet M$. 
\end{remark}

\begin{definition}[Underlying objects]
    Suppose $\CC$ admits countable colimits. Then $\const\colon\CC\to\Fil(\CC)$ has an obvious left adjoint, given by
    $$\fil^\bullet M\mapsto\colim_{n\to-\infty}\fil^nM.$$
    We call it the \emph{underlying object} of $\fil^\bullet M$ and denote it by $M$. 
\end{definition}

\begin{definition}[Day convolution]\label{day}
    Suppose $\CC$ admits countable colimits, and let $\otimes$ be a symmetric monoidal structure on it compatible with countable colimits. Then there is a natural symmetric monoidal structure on $\Fil(\CC)$ called \emph{Day convolution}, defined as
    $$\fil^r(M\otimes N)=\colim_{m+n\ge r}(\fil^mM\otimes\fil^nN),$$
    where the colimit is taken over the full subcategory of $\Z^\op\times\Z^\op$ spanned by $\{(m,n)\mid m+n\ge r\}$. For a rigorous $\infty$-categorical definition, see \cite[\S 2.2.6]{ha}. 

    By our compatibility assumption, it is easy to see that taking underlying objects is symmetric monoidal, so our notation above is compatible. It is also worth noting that taking associated graded objects is symmetric monoidal with respect to the tensor product
    $$(X\otimes Y)_r=\bigoplus_{m+n=r}X_m\otimes Y_n$$
    for graded objects. 
\end{definition}

\begin{definition}[Completion]\label{dfcpl}
    Suppose $\CC$ admits countable limits. We say that a filtered object $\fil^\bullet M$ of $\CC$ is \emph{complete} if
    $$\lim_{n\to+\infty}\fil^nM=0.$$
    We denote the full subcategory of complete filtered objects of $\CC$ by $\widehat{\Fil}(\CC)$, and write $\widehat{\DF}(-)$ for $\widehat{\Fil}(\D(-))$ as above. The inclusion functor $\widehat{\Fil}(\CC)\to\Fil(\CC)$ has an obvious left adjoint, given by
    $$\fil^\bullet M\mapsto\frac{\fil^\bullet M}{\lim_{n\to+\infty}\fil^nM},$$
    called \emph{completion}. Note that a filtered object is constant if and only if its completion is zero, and the completion functor is a localization with kernel the constant objects. 
\end{definition}

\begin{remark}[Monoidality of completion]
    Suppose $(\CC,\otimes)$ satisfies the assumptions of both Definition \ref{day} and Definition \ref{dfcpl}. Obviously, constant objects form a tensor ideal under Day convolution; therefore, by \cite[Proposition 2.2.1.9]{ha} there is a natural symmetric monoidal structure on $\widehat{\Fil}(\CC)$, where the completion functor is symmetric monoidal, and the inclusion functor is lax symmetric monoidal. 
\end{remark}

\begin{definition}[Trivial filtration]
    For $M\in\CC$, we call the filtered object
    $$\begin{array}{cc}
    \fil^nM=\begin{cases}M,&m\le0;\\0,&m>0;\end{cases}&
    \end{array}$$
    \emph{$M$ with trivial filtration}, and by abuse of notation we often denote it by $M$. This is natural in $M$, giving rise to a functor $\triv\colon\CC\to\Fil(\CC)$, which is easily seen to be the left adjoint of taking $\fil^0$. When $\CC$ satisfies the assumption of Definition \ref{day}, it is also easy to see that $\triv$ is symmetric monoidal. 
\end{definition}

\begin{definition}[Translation]
    For a filtered object $\fil^\bullet M$, its \emph{translation by $n$} is defined as $\fil^{\bullet+n}M$. Sometimes we also denote this translation functor by $-(n)$. 
\end{definition}

\begin{definition}[Filtered $\mathbb{E}_\infty$-algebras]
    Suppose $(\CC,\otimes)$ satisfies the assumption of Definition \ref{day}. Then a \emph{filtered $\mathbb{E}_\infty$-algebra} in $\CC$ is defined as an $\mathbb{E}_\infty$-algebra in $\Fil(\CC)$ with respect to Day convolution. By the above, one can take underlying $\mathbb{E}_\infty$-algebras, associated graded $\mathbb{E}_\infty$-algebras, and completions of filtered $\mathbb{E}_\infty$-algebras, and make trivially filtered $\mathbb{E}_\infty$-algebras starting with $\mathbb{E}_\infty$-algebras in $\CC$. 

    For a filtered $\mathbb{E}_\infty$-algebra $R$, we often call its modules in $\Fil(\CC)$ \emph{filtered $R$-modules}, and denote the category of them by $\DF(R)$. This is compatible with the notation in \cite[\S 5.1]{bms2}, as when $R$ is an $\mathbb{E}_\infty$-algebra in $\CC$ that is trivially filtered, $\DF(R)$ is just $\Fil(\D(R))$, where $\D(R)$ denotes the category of $R$-modules in $\CC$. 
\end{definition}

\begin{definition}[Filtered rings]
    Using \cite[Definition 4.2.22, Construction 4.3.4]{rak20}, we define the category of \emph{derived filtered rings} to be $\mathsf{DAlg}(\DF(\Z))$, and that of \emph{animated filtered rings} to be the full subcategory spanned by those derived filtered rings $R$ with $\fil^nR$ connective for all $n$. By \cite[Remark 4.2.24]{rak20}, one can alternatively define the category of animated filtered rings to be the animation, namely $\mathsf{P}_\Sigma$, of the $1$-category of classical filtered rings of the form $\lsym(M)$, where $M$ runs through finite direct sums of translations of the trivially filtered $\Z$. Passing through the Rees algebra construction, one can also view an animated filtered ring as an animated graded algebra over $\Z[t]$ (cf.\ Definition \ref{dtnc} below). 
\end{definition}

\begin{remark}
    By our convention, ``filtered rings'' will mean animated filtered rings. However, this filtered animation is never really used in this section and only rarely used in \S\ref{cycle}. Readers can safely replace filtered rings with filtered $\mathbb{E}_\infty$-rings in this section if they are willing to. 
\end{remark}

\subsection{Generalities on filtered modules}
In this subsection we discuss general properties of filtered $A$-modules that do not depend on the $\delta$-ring structure on $A$. 

We view both $A$ and $A[1/I]$ as filtered rings with the $I$-adic filtrations, i.e.\ 
$$\begin{array}{cc}
\fil^mA=\begin{cases}I^m,&m\ge0;\\A,&m<0;\end{cases}&
\fil^mA[1/I]=I^m.
\end{array}$$
Accordingly, $\DF(A)$ and $\DF(A[1/I])$ refer to the filtered derived categories with respect to the $I$-adic filtrations. Note that we are not breaking our convention that everything is $(p,I)$-complete; $A[1/I]$ is just a filtered $(p,I)$-complete $A$-algebra with underlying object $0$. We view $\bar{A}$ as with the trivial filtration, i.e.\
$$\fil^m\bar{A}=\begin{cases}\bar{A},&m\le0;\\0,&m>0;\end{cases}$$
then the reduction map $A\to\bar{A}$ is naturally a filtered map. 

\begin{proposition}\label{f0e}
Taking $\fil^0$ gives an equivalence $\DF(A[1/I])\cong\D(A)$, whose inverse is $M\mapsto I^\bullet M$.
\end{proposition}

\begin{proof}
$M\mapsto I^\bullet M$ is obviously right inverse to $\fil^0(-)$. Conversely, for $\fil^\bullet N\in\DF(A[1/I])$, from the filtered multiplication map $I^\bullet\otimes\fil^\bullet N\to\fil^\bullet N$ one gets maps $I^m\otimes\fil^n N\to\fil^{m+n}N$ for all $m,n\in\Z$. By their compatibility one can easily see that they are all isomorphisms, so $\fil^\bullet N=I^\bullet\fil^0N$. 
\end{proof}

Now we introduce two specializations of filtered $A$-modules, both of which take values in filtered $\bar{A}$-modules. 

\begin{definition}[Hodge--Tate specialization]\label{htf}
For $\fil^\bullet N\in\DF(A)$, the associated graded object $\gr^\bullet N$ is a graded module over $\gr^\bullet A=\bigoplus_{m\in\N}I^m/I^{m+1}$. Twist the \th{$m$} grade by $(I/I^2)^{\otimes(-m)}$; the module $\gr^\bullet N\{-\bullet\}$ we get is a graded module over $\gr^\bullet A\{-\bullet\}=\bar{A}[u]$ where $u$ is in the $1$\textsuperscript{st} grade. Note that such a graded module corresponds to an increasingly filtered $\bar{A}$-module, since multiplying by $u$ gives us the maps. This increasingly filtered $\bar{A}$-module is called the \emph{Hodge--Tate specialization} of $N$, denoted by $\fil_\bullet N^\HT\in\DF(\bar{A})$. 
\end{definition}

\begin{definition}[De Rham specialization]\label{drf}
The \emph{de Rham specialization} of $\fil^\bullet N\in\DF(A)$ is its base change along the filtered map $A\to\bar{A}$, denoted by $\fil^\bullet\bar{N}\in\DF(\bar{A})$. 
\end{definition}

\begin{remark}\label{grcp}
Since taking associated graded objects commutes with filtered base change, and the associated graded of the map $A\to\bar{A}$ is the map $\bigoplus_{m\in\N}I^m/I^{m+1}\to\bar{A}$ that kills all the $m>0$ summands, we can naturally identify the graded pieces of the Hodge--Tate specialization with those of the de Rham specialization up to twists by $I/I^2$. More specifically, for $N\in\DF(A)$, we have
$$\gr_\bullet N^\HT\{\bullet\}\cong\gr^\bullet\bar{N}.$$
In Example \ref{pc} below, this corresponds to the fact that the Hodge--Tate filtration on $\overline{\Prism_{R/A}}$ has graded pieces $\Omega_{R/\bar{A}}^\bullet\{-\bullet\}$ while the de Rham filtration on $\overline{\Prism_{R/A}^{(1)}}$ has graded pieces $\Omega_{R/\bar{A}}^\bullet$. 
\end{remark}

\begin{remark}[The underlying object of the Hodge--Tate specialization]\label{uht}
If we view a graded $\bar{A}[u]$-module $X_\bullet$ as an increasing filtration, then by definition its underlying object is the colimit of $X_n$ along multiplying $u$, which is the same as the \th{$0$} grade of $X_\bullet[u^{-1}]$. Accordingly, the underlying object $N^\HT$ of $\fil_\bullet N^\HT$ is the same as the \th{$0$} grade of the base change of $\gr^\bullet N$ from $\gr^\bullet A=\bigoplus_{m\in\N}I^m/I^{m+1}$ to $\gr^\bullet A[1/I]=\bigoplus_{m\in\Z}I^m/I^{m+1}$, which means that
$$N^\HT=\gr^0(N[1/I])=\fil^0(N[1/I])\otimes_A\bar{A},$$
where $\fil^\bullet N[1/I]$ denotes the filtered base change of $\fil^\bullet N$ from $A$ to $A[1/I]$, and the last equality comes from Proposition \ref{f0e}. 
\end{remark}

Finally we give some dualizability criteria. 

\begin{proposition}\label{cpdl}
Let $\CC$ be a presentably symmetric monoidal category. Then: 
\begin{itemize}
    \item If $1_\CC$ is compact, then dualizable objects in $\CC$ are compact. 
    \item If $\CC$ is compactly generated by its dualizable objects, then compact objects in $\CC$ are dualizable. 
\end{itemize}
\end{proposition}

\begin{proof}
If $1$ is compact and $X\in\CC$ is dualizable with dual $X^\vee$, then since
$$\Hom(X,-)=\Hom(1,\inhom(X,-))=\Hom(1,X^\vee\otimes-)$$
commutes with filtered colimits, $X$ is compact. 

If $\CC$ is stable, then finite colimits of dualizable objects are dualizable: for finite direct sums it is obvious, and for cofibers one easily verify that $\cofib(X\to Y)^\vee=\fib(Y^\vee\to X^\vee)$. In general, retracts of dualizable objects are dualizable. Thus, if $\CC$ is compactly generated by its dualizable objects, then compact objects, as retracts of finite colimits of generators, are dualizable. 
\end{proof}

\begin{corollary}\label{fd}
Let $R$ be a filtered ring. Then in $\DF(R)$, dualizable objects coincide with compact objects. If moreover the filtration of $R$ is trivial, then they are exactly the complete filtered objects with only finitely many nonzero graded pieces that are all dualizable in $\D(R)$. 
\end{corollary}

\begin{proof}
Note that: 
\begin{itemize}
    \item $1_{\DF(R)}=R$ and $\Hom(R,-)=\fil^0(-)$ commutes with colimits, so $1_{\DF(R)}$ is compact. 
    \item The translations $R(n)$ are all dualizable with $R(n)^\vee=R(-n)$, and form a family of compact generators of $\DF(R)$, since the functor $\DF(R)\to\Sp^\Z$, $\fil^\bullet N\mapsto(\fil^nN)_{n\in\Z}$ is clearly conservative. 
\end{itemize}
If the filtration of $R$ is trivial, then it is in particular complete, so every $R(n)$ is complete, with only one nonzero graded piece $\gr^{-n}R(n)=R$. Therefore, the compact objects in $\DF(R)$, as retracts of finite colimits of objects of the form $R(n)[d]$, are also complete, with only finitely many nonzero graded pieces that are all compact in $\D(R)$. Conversely, let $\fil^\bullet M\in\DF(R)$ be complete with only finitely many nonzero graded pieces that are all compact in $\D(R)$. Then for $m\gg0$ we have $\fil^mM=0$ and $\fil^{-m}M=M$. Also the terms $\fil^mM$ are all compact in $\D(R)$, being a finite extension of compact objects. Now it is easy to write $\fil^\bullet M$ as a finite extension of compact objects in $\D(R)$ with trivial filtration, which are clearly compact in $\DF(R)$. 
\end{proof}

\begin{lemma}\label{rfdl}
Let $\CC$ and $\DD$ be presentably symmetric monoidal categories and $F\colon\CC\to\DD$ be a conservative symmetric monoidal functor that preserves internal Homs. Then $c\in\CC$ is dualizable if and only if $Fc\in\DD$ is dualizable. 
\end{lemma}

\begin{proof}
If $c\in\CC$ is dualizable, clearly so is $Fc\in\DD$ since $F$ is symmetric monoidal. Conversely if $Fc\in\DD$ is dualizable, let $c^\vee=\inhom_\CC(c,1_\CC)$, then since $F$ is symmetric monoidal and preserves internal Homs, $F(c^\vee)=(Fc)^\vee$. To prove that $c\in\CC$ is dualizable, it suffices to show that the natural transformation $c^\vee\otimes_\CC-\to\inhom_\CC(c,-)$ is an isomorphism. This can be checked after applying $F$ by conservativity, whence it becomes obvious. 
\end{proof}

\begin{remark}\label{cam}
Lemma \ref{rfdl} applies when $A\in\CAlg(\CC)$ is dualizable in $\CC$, $\DD=\Mod_\CC(A)$ is the module category and the base change functor $F=-\otimes_\CC A$ is conservative, since
$$\inhom_\DD(X\otimes_\CC A,Y\otimes_\CC A)=\inhom_\CC(X,Y\otimes_\CC A)=\inhom_\CC(X,Y)\otimes_\CC A$$
by dualizability. 
\end{remark}

\begin{corollary}\label{abd}
An object in $\DF(\bar{A})$ is dualizable if and only if it is complete and has only finitely many nonzero graded pieces that are all dualizable in $\D(\bar{A})$. 
\end{corollary}

\begin{proof}
By Lemma \ref{rfdl}, $M\in\DF(\bar{A})$ is dualizable if and only if $M/p\in\DF(\bar{A}/p)$ is. Now the corollary follows from Corollary \ref{fd}. 
\end{proof}

\begin{proposition}\label{cons}
The functor $N\mapsto(\fil_\bullet N^\HT,\fil^\bullet\bar{N})$ from $\DF(A)$ to $\DF(\bar{A})\times\DF(\bar{A})$ is conservative.
\end{proposition}

\begin{proof}
Assume $\fil_\bullet N^\HT=\fil^\bullet\bar{N}=0$; we want to prove that $N=0$. This is easy: $\fil_\bullet N^\HT=0$ just means $\gr^\bullet N=0$ by definition, so the filtration is constant; then $\fil^\bullet\bar{N}$ is also constant, and that it is zero implies that $N$ is zero by $I$-completeness. 
\end{proof}

\begin{remark}
The conservativity above has a geometric interpretation: view filtered $A$-modules as graded modules over the Rees algebra $\R=\bigoplus_{m\in\Z}I^mt^{-m}$ where $t$ is in the $(-1)$\textsuperscript{st} grade. For simplicity assume $I=(d)$ is principal, so $\R=A[t,u]/(tu-d)$. Then the two specializations correspond to taking quotients by $t$ and $u$, so they are jointly conservative since $tu=d$ and everything is $d$-complete. 
\end{remark}

\begin{proposition}\label{dfad}
An object $N\in\DF(A)$ is dualizable if and only if its two specializations $\fil_\bullet N^\HT$ and $\fil^\bullet\bar{N}$ are both dualizable in $\DF(\bar{A})$, equivalently they are both complete and has only finitely many nonzero graded pieces that are all dualizable in $\D(\bar{A})$. 
\end{proposition}

\begin{proof}
It suffices to prove that the functor $N\mapsto(\fil_\bullet N^\HT,\fil^\bullet\bar{N})$ satisfies the assumptions of Lemma \ref{rfdl}. Note that the algebras $\bar{A}$ and $\gr^\bullet A$ in $\DF(A)$ are both dualizable, being cofibers of maps $I\to A$ and $A(1)\to A$. Hence the proposition follows from Remark \ref{cam} and Proposition \ref{cons}. 
\end{proof}

\subsection{\texorpdfstring{$F$}{F}-gauges over a prism}

\begin{definition}[Gauges and $F$-gauges]\label{deffg}
A \emph{gauge} over $(A,I)$ is a pair $(M,\fil^\bullet M^{(1)})$ where $M\in\D(A)$ and $\fil^\bullet M^{(1)}\in\DF(A)$ with underlying object $M^{(1)}=\varphi_A^*M$. An \emph{$F$-gauge} over $(A,I)$ is a triple $(M,\fil^\bullet M^{(1)},F)$ where $(M,\fil^\bullet M^{(1)})$ is a gauge over $(A,I)$ and $F\colon\fil^\bullet M^{(1)}[1/I]\to I^\bullet M$ is an isomorphism in $\DF(A[1/I])$. As above, $-[1/I]$ denotes the base change functor from $\DF(A)$ to $\DF(A[1/I])$ and $I^\bullet M$ denotes the obvious filtered module of $A[1/I]$ with underlying object $M[1/I]$.

Denote the categories of gauges and $F$-gauges over $(A,I)$ by $\Gauge(A,I)$ and $\FGauge(A,I)$, respectively. They are both presentably symmetric monoidal stable $\infty$-categories, since they are finite limits of such categories by definition. Clearly, the forgetful functor $\FGauge(A,I)\to\Gauge(A,I)$ is symmetric monoidal and preserves all colimits. 
\end{definition}

\begin{remark}\label{redeffg}
We often view $F$ as a morphism $F\colon\fil^\bullet M^{(1)}\to I^\bullet M$ in $\DF(A)$ that exhibits $I^\bullet M\in \DF(A[1/I])$ as the base change of $\fil^\bullet M^{(1)}\in\DF(A)$ to $A[1/I]$. Thus $\fil^\bullet M^{(1)}$ and $F$ actually determine $M$, as the $\fil^0$ of this base change. 
\end{remark}

\begin{remark}[Changing prisms]
Let $(A,I)\to(B,J)$ be a morphism of prisms. Then $(M,\fil^\bullet M^{(1)})\mapsto(M\otimes_AB,(\fil^\bullet M^{(1)})\otimes_{A^{(1)}}B^{(1)})$ gives base change functors $\Gauge(A,I)\to\Gauge(B,J)$ and $\FGauge(A,I)\to\FGauge(B,J)$. We denote them by $-\otimes_AB$ and call their right adjoints forgetful functors. Note that these forgetful functors commute with the forgetful functor $\D(B)\to\D(A)$ in the first component but not in the second, since for $N\in\D(B)$, $\varphi_B^*N$ and $\varphi_A^*N$ are in general different. 
\end{remark}

\begin{remark}[Descent]
Since each datum of a gauge or an $F$-gauge satisfies $(p,I)$-completely flat descent, so do gauges and $F$-gauges. In other words, $\Gauge$ and $\FGauge$ are sheaves of $\infty$-categories on the absolute prismatic site. 
\end{remark}

\begin{example}[Breuil--Kisin twists]\label{bk}
For $n\in\Z$, consider the Breuil--Kisin twist $A\{n\}$ defined in \cite[Definition 2.5.2, Notation 2.5.4]{apc}. It is an invertible $A$-module with canonical isomorphism $A\{n\}^{(1)}\cong I^{-n}A\{n\}$. We view it as an $F$-gauge over $(A,I)$ by defining
$$\fil^mA\{n\}^{(1)}=\begin{cases}
I^{m+n}A\{n\}^{(1)},&m\ge-n;\\
A\{n\}^{(1)},&m<-n;
\end{cases}$$
and taking $F$ to be the canonical isomorphism. Note that $\fil^mA\{n\}^{(1)}$ is chosen to correspond to $I^mA\{n\}$ under the canonical isomorphism for $m\ge-n$. 

Obviously, $A=A\{0\}$ is the tensor unit of $\FGauge(A,I)$, and $A\{n\}\otimes A\{n'\}=A\{n+n'\}$ for all $n,n'\in\Z$. By \cite[Remark 2.5.5]{apc}, the Breuil--Kisin twists commute with base changes, i.e.\ $A\{n\}\otimes_AB=B\{n\}$ for every prism $(B,J)$ over $(A,I)$.

For a gauge or an $F$-gauge $M$, let $M\{n\}$ denote the tensor product $M\otimes A\{n\}$ in the respective category. Then it is easy to see that
$$\fil^m(M\{n\}^{(1)})=(\fil^{m+n}M^{(1)})\{n\},$$
where the $\{n\}$ outside the brackets means the Breuil--Kisin twist relative to the animated prism $(A^{(1)},IA^{(1)})$. 
\end{example}

We discuss several specializations of $F$-gauges.

\begin{definition}[Hodge--Tate specialization]\label{hts}
For $M\in\Gauge(A,I)$, its \emph{Hodge--Tate specialization} is defined as the Hodge--Tate specialization of $M^{(1)}\in\DF(A)$ as in Definition \ref{htf}. For $M\in\FGauge(A,I)$, by Proposition \ref{f0e} and Remark \ref{uht}, the underlying object of its Hodge--Tate specialization is $\bar{M}=M\otimes_A\bar{A}$, so in this case we denote it by $\fil_\bullet\bar{M}$. 
\end{definition}

\begin{definition}[De Rham specialization]\label{drs}
For $M\in\Gauge(A,I)$, its \emph{de Rham specialization} is defined as the base change of $\fil^\bullet M^{(1)}\in\DF(A)$ to $\bar{A}$, denoted by $\fil^\bullet\overline{M^{(1)}}$. It is also a functor $\Gauge(A,I)\to\DF(\bar{A})$. 
\end{definition}

\begin{definition}[\'Etale specialization]\label{ets}
Here we suspend the convention that everything is $(p,I)$-complete. Suppose the prism $(A,I)$ is perfect with $R=A/p$, so $A=W(R)$ and $A[1/I]_p^\wedge=W(R[1/I])$. For $M\in\FGauge(A,I)$, since $M^{(1)}$ is just $M$ with the $A$-module structure twisted by Frobenius, we can view the isomorphism $F\colon M^{(1)}[1/I]_p^\wedge\to M[1/I]_p^\wedge$ as a Frobenius-semilinear automorphism of $M[1/I]_p^\wedge$. Now view $M[1/I]_p^\wedge$ as a sheaf on $\spa(R[1/I],R)_\proet\cong\spd(\bar{A}[1/p],\bar{A})_\proet$, and form the equalizer of $\id$ and $F$. This defines a functor $M\mapsto(M[1/I]_p^\wedge)^{F=1}: \FGauge(A,I)\to\D_\proet(\spd(\bar{A}[1/p],\bar{A}),\Z_p)$, which we call the \emph{\'etale specialization}.
\end{definition}

\begin{example}[Prismatic cohomology]\label{pc}
Let $X$ be a smooth formal scheme over $\spf(\bar{A})$. Then \cite[Theorem 1.8, Theorem 1.16]{bs19} gives the prismatic cohomology $\RG_\Prism(X/A)$ an $F$-gauge structure, whose Hodge--Tate specialization is the conjugate-filtered Hodge--Tate cohomology, whose de Rham specialization is the Hodge-filtered de Rham cohomology, and when $(A,I)$ is perfect whose \'etale specialization is the \'etale pushforward of $\Z_p$ along the diamond generic fiber of $X\to\spf(\bar{A})$. This can be generalized to general stacks over $\spf(\bar{A})$ using \cite[\S 5.1, \S 5.2]{apc}. 
\end{example}

\begin{remark}[Relative gauges and $F$-gauges]\label{relcoef}
It is possible to define for every formal scheme $X$ over $\spf(\bar{A})$ the categories $\Gauge(X/A)$ and $\FGauge(X/A)$ that serve as coefficients for the prismatic cohomology of $X$ over $(A,I)$: First for $X=\spf(R)$ with $R$ a large quasisyntomic $\bar{A}$-algebra as in \cite[Definition 15.1]{bs19}, define
$$\Gauge(X/A)=\{(M,\fil^\bullet M^{(1)})\mid M\in\D(\Prism_{R/A}),\,\fil^\bullet M^{(1)}\in\DF(\fil^\bullet\Prism_{R/A}^{(1)})\},$$
$$\FGauge(X/A)=\{(M,\fil^\bullet M^{(1)},F)\mid(M,\fil^\bullet M^{(1)})\in\Gauge(X/A),$$
$$F\colon\fil^\bullet M^{(1)}\otimes_{\Prism_{R/A}^{(1)},\varphi}\Prism_{R/A}[1/I]\cong M[1/I]\};$$
next for smooth affine $X$ use quasisyntomic descent from the large algebras; then for general affine $X$ use left Kan extension from the smooth ones; finally for general $X$ use descent again. 

$\FGauge(X/A)$ should have specialization to relative prismatic crystals (i.e.\ quasicoherent $\O_\Prism$-modules on the site $(X/A)_\Prism$), Hodge--Tate specialization to Higgs modules, de Rham specialization to $\DD$-modules, and when $(A,I)$ is perfect, \'etale specialization to \'etale sheaves on the generic fiber. Also, the Poincar\'e duality should hold for proper smooth $X$ with dualizable coefficients in these categories. However we will not explore these in this paper. 
\end{remark}

The following proposition computes global sections of gauges and $F$-gauges, with which we can reinterpret the relative syntomic cohomology $\Z_p(n)(X/A)$ in \cite{ktannounce} as $\Hom_{\FGauge(A,I)}(A,\RG_\Prism(X/A)\{n\})$, cf.\ Example \ref{bk} and Example \ref{pc}. 

\begin{proposition}\label{glsect}
If $M$ is a gauge, then
$$\Hom_{\Gauge(A,I)}(A,M)=M\times_{M^{(1)}}\fil^0M^{(1)},$$
where the map $M\to M^{(1)}=M\otimes_{A,\varphi}A$ is $m\mapsto m\otimes1$. If $M$ is an $F$-gauge, then
$$\Hom_{\FGauge(A,I)}(A,M)=\eq(\fil^0M^{(1)}\rightrightarrows M^{(1)}),$$
where the two maps are the natural map and the composition $$\fil^0M^{(1)}\to\fil^0(M^{(1)}[1/I])=M\to M^{(1)}.$$
\end{proposition}

\begin{proof}
The first sentence is clear from the definition of $\Gauge(A,I)$, and the second sentence is clear from Remark \ref{redeffg}. 
\end{proof}

We finally offer a dualizability criterion using Proposition \ref{dfad}. 

\begin{proposition}[Dualizability of gauges]\label{dual}
Let $M$ be a gauge or an $F$-gauge. Then $M$ is dualizable if and only if it satisfies the following conditions: 
\begin{enumerate}
    \item\label{dda} $M$ is dualizable in $\D(A)$. 
    \item\label{dht} The Hodge--Tate specialization of $M$ is dualizable, i.e.\ is complete with only finitely many graded pieces that are perfect complexes. 
    \item\label{ddr} The de Rham specialization of $M$ is dualizable, i.e.\ is complete with only finitely many graded pieces that are perfect complexes. 
\end{enumerate}
The perfectnesses of the graded pieces in (\ref{dht}) and (\ref{ddr}) are equivalent, and in the $F$-gauge case, (\ref{dht}) implies (\ref{dda}). 
\end{proposition}

\begin{proof}
Recall that the symmetric monoidal category of gauges is defined as the pullback
$$\begin{tikzcd}
\Gauge(A,I)\ar[r]\ar[d]&\DF(A)\ar[d]\\
\D(A)\ar[r]&\D(A)
\end{tikzcd}$$
where the right vertical arrow is taking underlying objects and the lower horizontal arrow is base change along $\varphi$. Therefore, $M\in\Gauge(A,I)$ is dualizable if and only if its two images in $\fil^\bullet M^{(1)}\in\DF(A)$ and $M\in\D(A)$ are dualizable. Now the proposition follows from Proposition \ref{dfad}. The two perfectnesses are equivalent because of Remark \ref{grcp}. 

Recall that by Proposition \ref{f0e}, the symmetric monoidal category of $F$-gauges is defined as the equalizer
$$\FGauge(A,I)\to\Gauge(A,I)\rightrightarrows\D(A)$$
where the two arrows send a gauge $(M,\fil^\bullet M^{(1)})\in\Gauge(A,I)$ to $M\in\D(A)$ and $\fil^\bullet M^{(1)}[1/I]\in\DF(A[1/I])\cong\D(A)$, respectively. Therefore, an $F$-gauge $M$ is dualizable if and only if it is dualizable as a gauge. Since $\bar{M}\in\D(\bar{A})$ is the underlying object of the Hodge--Tate specialization, the dualizability of the Hodge--Tate specialization implies that of $\bar{M}$, which in turn implies that of $M$ by Lemma \ref{rfdl}. 
\end{proof}

\begin{corollary}\label{cohdual}
Let $X$ be a proper smooth formal algebraic space over $\spf(\bar{A})$. Then $\RG_\Prism(X/A)\in\FGauge(A,I)$ is dualizable. 
\end{corollary}

\begin{proof}
By Proposition \ref{dual} we only need to show that the Hodge--Tate specialization and the de Rham specialization of $\RG_\Prism(X/A)$ are both dualizable. This is clear: both are complete with the graded pieces zero outside $[0,\dim(X)]$; the nonzero graded pieces are $\RG(X,\Omega^\star_{X/\bar{A}}[-\star])$, which is dualizable by \cite[\texttt{0A1P}]{stacks}. 
\end{proof}

\section{Weighted homotopy invariance of cohomologies}

One remarkable property that the $\ell$-adic \'etale cohomology enjoys is the $\A^1$-invariance, namely $\RG_\et(\A^1_X,\Z_\ell)=\RG_\et(X,\Z_\ell)$ for $\ell$ invertible on $X$. Not only does this fact help us compute the \'etale cohomology for many schemes, it has also been a crucial input to the norm residue isomorphism theorem of Suslin and Voevodsky; see for example \cite{vsbk}. 

However, various cohomologies studied in $p$-adic Hodge theory, including Hodge, de Rham, prismatic, and syntomic, usually do not satisfy the $\A^1$-invariance. Fortunately, as \cite{dr} suggests, they actually enjoy a weaker property, the \emph{weighted $\A^1$-invariance}, namely $\FF(X\times\A^1/\G_m)=\FF(X\times\B\G_m)$ for $\FF$ being some of the above cohomologies. In this section, we will formulate this invariance property in general, and prove it for some cohomologies, in order for later sections to exploit. 

\subsection{Big \'etale sheaves}
We first study the notion of big \'etale sheaves. 

\begin{definition}[Big \'etale sheaves]
Let $\CC$ be a presentable $\infty$-category. A \emph{big \'etale sheaf} with values in $\CC$ is an accessible functor $\FF\colon\Ring\to\CC$ that satisfies \'etale descent, i.e.\ a $\CC$-object of the big \'etale topos. A \emph{stack} in this paper is a big \'etale sheaf of animas. For a stack $S$, a \emph{big \'etale sheaf over $S$} with values in $\CC$ is an accessible functor $\FF\colon\Ring_{/S}\to\CC$ that satisfies \'etale descent, i.e.\ a $\CC$-object of the slice topos over $S$. A big \'etale sheaf of animas over $S$ is the same as a stack with a map to $S$, and thus is called a \emph{stack over $S$}. 
\end{definition}

\begin{remark}[Small \'etale sheaves]
    Let $A$ be a ring and let $X$ be a sheaf of animas on the small \'etale site $A_\et$, i.e.\ the \'etale site of \'etale $A$-algebras. Then one can view it as a stack over $A$ by pulling back along the topos map $(\Ring_{/A})_\et\to A_\et$. Such a stack is called an \emph{\'etale stack over $A$}. We call a map of stacks \emph{\'etale} if it is \'etale after every pullback to a representable. 
\end{remark}

\begin{example}
All geometric objects that appear in this paper are stacks, including: 
\begin{itemize}
    \item Rings, i.e.\ affine schemes. For a ring $A$, the functor $R\mapsto\Hom(A,R)$ is a stack, denoted $\spec(A)$. We often identify $A$ and $\spec(A)$. 
    \item Deligne--Mumford stacks. They are just stacks that are affine schemes \'etale locally. Namely, a Deligne--Mumford stack is a stack $X$ that is covered by a family of \'etale maps from affine schemes, cf.\ \cite[\S 1.6.4]{sag}. 
    \item Algebraic spaces. They are just those Deligne--Mumford stacks that take classical rings to sets, cf.\ \cite[Definition 1.6.8.1]{sag}. 
    \item Affine formal schemes. For a ring $A$ and a finitely generated ideal $I\subseteq\pi_0A$, the functor 
    $$R\mapsto\{f\in\Hom(A,R)\mid\text{$\pi_0(f)(I)$ is nilpotent}\}$$
    is a stack, denoted $\spf(A,I)$, or $\spf(A)$ if not confusing. We often identify $A$ and $\spf(A)$ whenever $A$ is complete, so $\D(A)$ will mean the category of derived $I$-complete complexes, and in particular this applies to $\D(\Z_p)$.  
    \item Formal Deligne--Mumford stacks. They are just stacks that are affine formal schemes \'etale locally, in the same way as above. 
    \item Formal algebraic spaces. They are just formal Deligne--Mumford stacks that take classical rings to sets. 
    \item Quotient stacks. Let $X$ be a stack and $G$ be a group scheme smooth over $\Z$ acting on $X$. Then 
    $$X/G=(R\mapsto\{(E,f)\mid\text{$E\to\spec(R)$ is a $G$-torsor, $f\colon E\to X$ is $G$-equivariant}\})$$
    is a stack that receives an effective epimorphism from $X$, since smooth maps have sections \'etale locally. Alternatively, $X/G=\colim_{n\in\Delta^\op}(X\times G^n)$, where the simplicial object is defined by the group action. 
\end{itemize}
\end{example}

\begin{remark}[Taking values in stacks]\label{valsta}
Let $\CC$ be a presentable category and let $S$ be a stack. Let $\FF$ be a big \'etale sheaf over $S$ and let $X$ be a stack over $S$. Viewing $X$ as a big \'etale sheaf over $S$ with values in $\Ani$, since $\CC$ is naturally powered over $\Ani$, the anima $\Hom_S(X,\FF)$ is naturally enhanced to an object $\FF(X)\in\CC$, which we call the \emph{value of $\FF$ on $X$}. Equivalently, 
$$\FF(X)=\lim_{\spec(R)\to X}\FF(R).$$
Note that $\FF(X)$ takes colimits in $X$ to limits, since it is $\Hom$ out of it. 
\end{remark}

The following notion is central in this section. 

\begin{definition}[Weighted homotopy invariance]\label{wtinv}
We say that a big \'etale sheaf $\FF$ over a stack $S$ satisfies \emph{strong weighted homotopy invariance} if, for every $\N$-graded ring $R=\bigoplus_{n\in\N}R_n$ over $S$, viewed as acted by $\G_m$ naturally by grading, 
$$\FF(\spec(R)/\G_m)=\FF(\spec(R_0)/\G_m).$$
In this case, we have in particular $\FF(V/\G_m)=\FF(X/\G_m)$ for any stack $X$ and vector bundle $V$ over it, where $\G_m$ acts on $X$ trivially and acts on $V$ by scaling. This follows from take the limit of the definition over all $\spec(R_0)\to X$. Therefore, we say that $\FF$ satifies \emph{smooth weighted homotopy invariance} if, for every vector bundle $V\to X$ with $X$ smooth over $S$, we have $\FF(V/\G_m)=\FF(X/\G_m)$. 
\end{definition}

\begin{remark}[Quasicoherent sheaves on stacks]\label{qcoh}
Let $\Pr_\st^\L$ denote the category of presentable stable categories with morphisms the left adjoints. Then $R\mapsto\D(R)$ satisfies \'etale descent, and is nearly a big \'etale sheaf except that $\Pr_\st^\L$ is itself not presentable. Yet $\Pr_\st^\L$ admits arbitrary small limits, so Remark \ref{valsta} gives rise to a presentable stable category $\D(X)$ for any stack $X$, called the \emph{derived category of $X$}, and its object is called a \emph{quasicoherent sheaf on $X$}. By construction, for a stack map $f\colon X\to S$, we have the adjoint pair $(f^*,f_*)\colon\D(S)\rightleftarrows\D(X)$. We sometimes write $\RG_{X/S}$ instead of $f_*$.  
\end{remark}

\begin{remark}[Small \'etale $\Z_p$-sheaves on stacks]
Similarly, $R\mapsto\D_\et(\spec(R),\Z_p)$ satisfies \'etale descent, so Remark \ref{valsta} gives rise to a presentable stable category $\D_\et(X,\Z_p)$ for any stack $X$, whose object is called an \emph{\'etale $\Z_p$-sheaf on $X$}. Here we also have an adjoint pair $(f^*,f_*)\colon\D_\et(S,\Z_p)\rightleftarrows\D_\et(X,\Z_p)$ for every stack map $X\to S$. Alternatively, $\D_\et(-,\Z_p)$ is the big \'etale sheafification of the constant functor $\Ring\to\Pr_\st^\L$ with value $\D(\Z_p)$. 
\end{remark}

\begin{example}[\'Etale cohomology]\label{etcoh}
Let $S$ be a stack and let $\FF$ be an \'etale $\Z_p$-sheaf on $S$. Then the \emph{\'etale cohomology} $X\mapsto\RG_\et(X,f^*\FF)$ is a big \'etale sheaf over $S$ with values in $\D(\Z_p)$, where $f\colon X\to S$ is the structure map. By abuse of notation, we denote this big \'etale sheaf still by $\FF$. Some important special cases of this construction are:
\begin{enumerate}
    \item $S=\spec(\Z[1/p])$ and $\FF=\Z_p(n)$ is the Tate twist, where $n\in\Z$.
    \item\label{j!Zp} $S=\spec(\Z)$ and $\FF=j_!\Z_p(n)$ where $j\colon\spec(\Z[1/p])\to\spec(\Z)$. 
\end{enumerate}
\end{example}

\begin{example}[\'Etale cohomology of the adic generic fiber]\label{etadic}
Let $S$ be a stack over $\spec(\Z[t])$ and let $\FF$ be an \'etale $\Z_p$-sheaf on $S$. Then the \emph{\'etale cohomology of the adic generic fiber} $R\mapsto\RG_\et(\hat{R}_t[1/t],f^*\FF)$ is a big \'etale sheaf over $S$ with values in $\D(\Z_p)$, where $\hat{R}_t$ is the $t$-completion of $R$ and $f\colon\spec(\hat{R}_t[1/t])\to S$ is the obvious map. In fact, by \cite[Corollary 6.17]{arc}, it is even an $\arc_t$ sheaf. The most important special case is $S=\spec(\Z)$ with structure map $t\mapsto p$. 
\end{example}

\begin{example}[Hodge cohomology]\label{hodgecoh}
Let $S$ be a stack. The \emph{Hodge cohomology} $X\mapsto\RG_{X/S}(\LO_{X/S}^\star)[-\star]$ is a big \'etale sheaf over $S$ with values in $\Gr(S)$. We denote it by $\LO_{-/S}^\star[-\star]$. When $S=\spec(\Z)$ we often omit the ``$/S$''. 
\end{example}

\begin{example}[De Rham cohomology]
Let $S$ be a stack. The \emph{de Rham cohomology} $X\mapsto\fil^\bullet\dr_{X/S}$ is a big \'etale sheaf over $S$ with values in $\DF(S)$. Note that its associated graded object is the Hodge cohomology. 
\end{example}

\begin{example}[Relative prismatic cohomology]
Let $(A,I)$ be a prism and let $\bar{A}=A/I$. The \emph{relative prismatic cohomology} $R\mapsto\Prism_{R/A}$ is a big \'etale sheaf over $\bar{A}$ with values in $\FGauge(A,I)$. Its Hodge--Tate specialization $\fil_\bullet\overline{\Prism}_{-/A}$ (Definition \ref{hts}) is called the \emph{relative Hodge--Tate cohomology}, and its de Rham specialization (Definition \ref{drs}) is the de Rham cohomology over $\spf(\bar{A})$. 
\end{example}

\begin{example}[Absolute prismatic cohomology]
Let $\wcart$ and $\wcart^\HT$ be the Cartier--Witt stack and its Hodge--Tate divisor, as in \cite[\S 3.3,\S 3.4]{apc}. The \emph{prismatic cohomology sheaf} $\HH_\Prism$ defined in \cite[Construction 4.4.1]{apc} is a big \'etale sheaf with values in $\D(\wcart)$, and the (twisted and Nygaard filtered) \emph{absolute prismatic cohomology} is the sheaf $\Prism_-\{n\}=\RG(\wcart,\HH_\Prism\{n\})\in\DF(\Z_p)$ defined in \cite[Construction 4.4.10]{apc} with the Nygaard filtration $\fil^\bullet\Prism_-\{n\}$ as in \cite[Construction 5.5.3]{apc}, for $n\in\Z$. The (conjugate filtered) \emph{Hodge--Tate cohomology sheaf} $\fil_\bullet\HH_{\overline{\Prism}}$ defined in \cite[Construction 4.5.1]{apc} is a big \'etale sheaf with values in $\DF(\wcart^\HT)$, and similarly the (twisted and conjugate filtered) \emph{absolute Hodge--Tate cohomology} is the sheaf $\fil_\bullet\overline{\Prism}_-\{n\}=\RG(\wcart^\HT,\fil_\bullet\HH_{\overline{\Prism}}\{n\})\in\DF(\Z_p)$, for $n\in\Z$. 
\end{example}

\begin{example}[Diffracted Hodge cohomology]\label{difhodge}
The \emph{diffracted Hodge cohomology} $\fil_\bullet\OD_-$, defined in \cite[\S 4.9]{apc}, is a big \'etale sheaf with values in $\DF(\B\G_m^\sharp)$. Rationally it is the Hodge cohomology with the filtration na\"ively defined by the grading, under the identifications $\G^\sharp_{m,\Q}=\G_{m,\Q}$ and $\D(\B\G_{m,\Q})=\Gr(\Q)$, while after $p$-completion it is the Hodge--Tate cohomology sheaf, under the identification $\B\G^\sharp_{m,\Z_p}=\wcart^\HT$ (see \cite[Theorem 3.4.13]{apc}). Its associated graded object is the Hodge cohomology, where the $\G_m^\sharp$-action on the \th{$n$} grade comes from the $\G_m$-action of weight $-n$. 
\end{example}

\begin{example}[Syntomic cohomology]\label{syncoh}
The \emph{syntomic cohomology} $\Z_p(\star)$, defined in \cite[\S 8.4]{apc}, is a big \'etale sheaf with values in $\Gr(\Z_p)$. By \cite[Construction 7.4.1]{apc}, for $n\in\Z$ and $X$ over $\spf(\Z_p)$, $\Z_p(n)(X)$ is the equalizer of the maps $\varphi\{n\}$ and $\iota$ from $\fil^n\Prism_X\{n\}$ to $\Prism_X\{n\}$, while by \cite[Remark 8.4.4]{apc}, for a stack $X$ with $p$-completion $\hat{X}$, there is a fiber sequence
$$j_!\Z_p(n)(X)\to\Z_p(n)(X)\to\Z_p(n)(\hat{X})$$
where the first term is as in Example \ref{etcoh}. 
\end{example}

\begin{remark}
As \cite{apc} has constructed, all the examples above except \ref{etcoh}(\ref{j!Zp}) have symmetric algebra structures in their respective categories, i.e.\ they actually have values in $\CAlg(\Z_p)$, $\CAlg\DF(\Z_p)$, and so on. We have omitted this clumsy ``$\CAlg$'' notation in above.
\end{remark}

\begin{remark}
Almost every cohomology theory that we encounter is left Kan extended from its restriction to smooth $\Z$-algebras, at least when formulated properly. See \cite[Proposition 4.5.8, Proposition 4.5.10, Proposition 4.9.6, Remark 5.5.10, Proposition 8.4.10]{apc}. Therefore, our generality of allowing arbitrary animated rings to be ``test objects'' of stacks may be redundant; in most cases smooth $\Z$-algebras suffice. However, in order not to introduce more complication, we refrain from restricting test objects. 
\end{remark}

\subsection{Comparing cohomology}
This subsection aims to abstract the arguments of \cite[Proposition 9.2.9]{apc} and \cite[\S 9.4]{apc}. We first generalize the notion of quasisyntomic rings in \cite[\S 4.2]{bms2} to animated rings. Before that let us recall the notion of $\tor$ amplitude. Note that unlike \cite[\S 4.1]{bms2}, we use the homological indexing here. 

\begin{definition}[$\tor$ amplitude]
Let $A$ be a ring and let $M\in\D(A)$. For $a,b\in\Z$, we say that the \emph{($p$-complete) $\tor$ amplitude} of $M$ is in $[a,b]$ if, for every classical $A$-module $N$ (with $pN=0$), $M\otimes_AN$ has nonzero homotopy groups only in degree $[a,b]$. Similarly we have the notion of ($p$-complete) $\tor$ amplitude in $\ge a$ or $\le b$. If an $A$-module $M$ has ($p$-complete) $\tor$ amplitude concentrated at $0$, we say that it is \emph{($p$-completely) flat}. 
\end{definition}

Obviously, ($p$-complete) $\tor$ amplitude is preserved by arbitrary base change. 

\begin{proposition}\label{torampbound}
Let $n\in\N$ and let $A$ be a ring whose homotopy groups of degrees $>n$ are zero. Let $a,b\in\Z$ and let $M\in\D(A)$ have $\tor$ amplitude in $[a,b]$ (resp.\ $\ge a$, $\le b$). Then $M$ has nonzero homotopy groups only in degree $[a,b+n]$ (resp.\ $\ge a$, $\le b+n$). 
\end{proposition}

\begin{proof}
This is simply because $A$ itself is an iterated extension of $(\pi_iA)[i]$ for $i=0,1,\ldots,n$, and each $\pi_iA$ is a classical $A$-module. 
\end{proof}

The following lemmas discuss commutation of cosimplicial totalizations and filtered colimits, under some coconnectivity conditions. Let $I$ be a filtered index category and $\Delta$ be the opposite category of ordered nonempty finite sets, which is also called the category of standard simplices. 

\begin{lemma}\label{colimlimcoco}
Let $M$ be a diagram of spectra indexed by $\Delta\times I$. Suppose $\tau_{>0}M$ is ind-zero as a diagram $I\to\Fun(\Delta,\Sp)$, i.e.\ for every $i\in I$, there is a map $i\to j$ in $I$ such that the map $\tau_{>0}M_i\to\tau_{>0}M_j$ is zero in $\Fun(\Delta,\Sp)$. Then
$$\colim_I\lim_{\Delta}M=\lim_{\Delta}\colim_IM$$
and is coconnective. 
\end{lemma}

\begin{proof}
There is an obvious map from the left hand side to the right hand side. The fiber sequence $\tau_{>0}M\to M\to\tau_{\le0}M$ of diagrams induces fiber sequences for both sides. The assumption implies that
$$\colim_I\lim_{\Delta}\tau_{>0}M=\lim_{\Delta}\colim_I\tau_{>0}M=0.$$
Therefore, one can assume that $M=\tau_{\le0}M$ is coconnective. For $n\in\N$, let $M_{\le n}$ be the restriction of $M$ to $\Delta_{\le n}\times I$, where $\Delta_{\le n}$ is the full subcategory of $\Delta$ generated by the simplices of dimension $\le n$. Then for every $n\in\N$: 
\begin{itemize}
    \item For every $i\in I$ we have 
    $$\fib(\lim_{\Delta}M_i\to\lim_{\Delta_{\le n}}M_{\le n,i})\in\Sp_{\le-n};$$
    since filtered colimits respect coconnectivity, we can take $\colim_I$ and get
    $$\fib(\colim_I\lim_{\Delta}M\to\colim_I\lim_{\Delta_{\le n}}M_{\le n})\in\Sp_{<-n}.$$
    \item Also by the fact that filtered colimits respect coconnectivity,  
    $$\fib(\lim_{\Delta}\colim_IM\to\lim_{\Delta_{\le n}}\colim_IM_{\le n})\in\Sp_{<-n}.$$
    \item Since $\Delta_{\le n}$ is finite and in a stable $\infty$-category finite limits are just shifts of finite colimits,
    $$\colim_I\lim_{\Delta_{\le n}}M_{\le n}=\lim_{\Delta_{\le n}}\colim_IM_{\le n}.$$
\end{itemize}
Combining, we get
$$\fib(\colim_I\lim_{\Delta}M\to\lim_{\Delta}\colim_IM)\in\Sp_{\le-n},$$
but the fiber is independent of $n$, so it is zero, and the lemma follows. 
\end{proof}

\begin{lemma}
Let $A$ be a homotopically bounded ring, i.e.\ $\pi_nA=0$ for $n\gg0$. Let $M$ be a diagram of objects in $\D(A)$ with $\tor$ amplitude $\le 0$, indexed by $\Delta\times I$. Then
$$\colim_I\lim_{\Delta}M=\lim_{\Delta}\colim_IM.$$
\end{lemma}

\begin{proof}
Combine Proposition \ref{torampbound} and (some shift of) Lemma \ref{colimlimcoco}. 
\end{proof}

The terms below are taken from \cite[\S 4.1]{evfil}. Their apparent incompatibility originates from \cite[Definition 4.10]{bms2}. 

\begin{definition}
Let $A\to R$ be a ring map. We say that it is \emph{($p$-)quasi-lci} if $\L_{R/A}\in\D(R)$ has ($p$-complete) $\tor$ amplitude $\le1$. If it is moreover ($p$-completely) flat, we say that it is \emph{($p$-)quasisyntomic}. A \emph{($p$-)quasisyntomic} ring means a ring that is ($p$-)quasi-lci over $\Z$. 
\end{definition}

If $A\to R$ is a ($p$-)quasi-lci ring map, since $\L_{R/A}$ is connective, $\L_{R/A}$ actually has ($p$-complete) $\tor$ amplitude in $[0,1]$. Therefore, its wedge product $\LO^n_{R/A}$ has ($p$-complete) $\tor$ amplitude in $[0,n]$, for all $n\in\N$. 

Now we begin to compare the cohomology of different stacks. In what follows, let $(X_i)_{i\in I}$ be a finite diagram of stacks, and let $X$ be a stack receiving compatible maps from them. Rigorously speaking, the data should be a diagram $(X_i)_{i\in I\star\Delta^0}$, where $X=X_{\Delta^0}$. Then for any big \'etale sheaf $\FF$ with values in any presentable category $\CC$, we have a natural map
\begin{equation}\label{maptolim}
    \FF(X)\to\lim_{i\in I}\FF(X_i).
\end{equation}
The following lemmas discuss conditions under which the map is an isomorphism. 

\begin{lemma}\label{limiso}
If the map (\ref{maptolim}) is an isomorphism for $\FF=\LO^\star$, then so is it for $\FF(R)=\RG_\syn(\spf(\hat{R}),\Z_p(\star))$. If it is moreover an isomorphism either for $\FF=j_!\Z_p(\star)$ or for $\FF$ being both $R\mapsto\RG_\et(R[1/p],\Z_p(\star))$ and $R\mapsto\RG_\et(\hat{R}[1/p],\Z_p(\star))$, then so is it for $\FF=\Z_p(\star)$. 
\end{lemma}

\begin{proof}
We prove the lemma in steps: 
\begin{itemize}
    \item For $\FF=\fil_\bullet\OD_-$, (\ref{maptolim}) is an isomorphism. This is because $\fil_{<0}\OD_-=0$ and $\gr_n\OD_-=\LO_-^n[-n]$, by \cite[Remark 4.9.3]{apc}. 
    \item For $\FF=\gr^\bullet\Prism_-\{\star\}$, (\ref{maptolim}) is an isomorphism. This is due to the fiber sequence
    $$\gr^m\Prism_R\{n\}\to\fil_m\widehat{\OD_R}\to\fil_{m-1}\widehat{\OD_R}$$
    in \cite[Remark 5.5.8]{apc}. Here the hat denotes $p$-completion. 
    \item For $\FF=\fil^\bullet\hat{\Prism}_-\{\star\}$, where the hat denotes Nygaard completion, (\ref{maptolim}) is an isomorphism. This follows from forming finite extensions of $\gr^\bullet\Prism_-\{\star\}$ and then taking limits. 
    \item Finally, the lemma follows from \cite[Construction 7.4.1, Proposition 7.4.6, Construction 8.4.1, Remark 8.4.4]{apc}. \qedhere
\end{itemize}
\end{proof}

\begin{lemma}\label{qsynlimiso}
Assume each of the $X_i$ and $X$ is a simplicial colimit of the form
$$\colim\left(\begin{tikzcd}
\cdots\ar[r,shift left=2]\ar[r]\ar[r,shift right=2]&\spec(R_1)\ar[l,shift left]\ar[l,shift right]\ar[r,shift left]\ar[r,shift right]&\spec(R_0)\ar[l]
\end{tikzcd}\right)$$
with the face maps all ($p$-completely) flat and the $R_m$ all ($p$-)quasisyntomic and ($p$-completely) homotopically bounded. Then if the map (\ref{maptolim}) is an isomorphism for $\FF=\LO^\star$, so is it for $\FF$ being (the $p$-completions of) $\OD_-$, $\HH_\Prism$, and $\fil^\bullet\Prism_-\{\star\}$. 
\end{lemma}

\begin{proof}
We first show that, for any stack $X=\colim_{m\in\Delta^\op}\spec(R_m)$ as assumed, we have $\OD_X=\colim_n\fil_n\OD_X$ ($p$-completely). For this, take $N\in\N$ such that $\pi_{>N}R_0=0$ ($p$-completely). Then by flatness, $\pi_{>N}R_n=0$ ($p$-completely) for every $n\in\N$. By Proposition \ref{torampbound} one can see that $\LO^n_{R_m}[-n]\in\Sp_{\le N}$ ($p$-completely), and thus so is $\fil_n\OD_{R_m}$. Now by Remark \ref{valsta} and Lemma \ref{colimlimcoco}, we have ($p$-completely)
\begin{align*}
    \OD_X=\lim_{m\in\Delta}\OD_{R_m}&=\lim_{m\in\Delta}\colim_n\fil_n\OD_{R_m}\\
    &=\colim_n\lim_{m\in\Delta}\fil_n\OD_{R_m}=\colim_n\fil_n\OD_X.
\end{align*}

With this in hand, we now prove the lemma: 
\begin{description}
    \item[$\OD_-$] This follows from the previous paragraph and the proof of Lemma \ref{limiso}. 
    \item[$\HH_\Prism$] This is because $\HH_{\overline{\Prism}}$ is the $p$-completion of $\OD_-$ (cf.\ Example \ref{difhodge}), and $\wcart$ is complete along $\wcart^\HT$. 
    \item[$\fil^\bullet\Prism_-\{\star\}$] For $\Prism_-\{\star\}$ this follows by twisting $\HH_\Prism$ by $\O_\wcart\{\star\}$ and taking global sections. For $\gr^\bullet\Prism_-\{\star\}$ this follows from the proof of Lemma \ref{limiso}. Combine them and get what we want. \qedhere
\end{description}
\end{proof}

\begin{remark}\label{qsyndescent}
The proof of Lemma \ref{qsynlimiso} shows that, for $X=\colim_{m\in\Delta^\op}\spec(R_m)$ as assumed and $N\in\N$ with $\pi_{>N}R_0=0$ ($p$-completely), the spectra $\OD_X$, $\HH_\Prism(X)$, $\fil^i\Prism_X\{n\}$, and $\RG_\syn(X,\Z_p(n))$ are in $\Sp_{\le N}$ ($p$-completely) for all $i,n\in\Z$, since the \'etale cohomology is always coconnective. 
\end{remark}

\begin{remark}
The proof of Lemma \ref{qsynlimiso} also shows that (the $p$-completions of) $\OD_-$, $\HH_\Prism$, and $\fil^\bullet\Prism_-\{\star\}$ satisfy descent along a ($p$-completely) quasisyntomic map between ($p$-)quasisyntomic and ($p$-completely) homotopically bounded rings. 
\end{remark}

\subsection{Cohomologies of graded rings}
In this subsection, we finally establish the strong weighted homotopy invariance of the syntomic cohomology. 

First we discuss the Hodge cohomology. The first part of the following definition is taken from \cite[Lecture 13]{dr}; beware that it is expected to be somewhat reasonable only when $Z$ is an Artin stack. 

\begin{definition}[Cotangent complex]
    For a stack $Z$ over a ring $A$, let
    $$\L_{Z/A}=\lim_{g\colon\spec(R)\to Z}g_*\L_{R/A}\in\D(Z),$$
    where $g_*$ denotes the quasicoherent pushforward as in Remark \ref{qcoh}. We often write $\L_Z$ for $\L_{Z/\Z}$. For a map of stacks $X\to Y$, let 
    $$\L_{X/Y}=\lim_{f\colon\spec(R)\to Y}(f_X)_*\L_{(X\times_Y\spec(R))/R}\in\D(X),$$
    where $f_X$ is the base change of $f$ along $X\to Y$. 
\end{definition}

\begin{remark}\label{computeL}
    Note that the assignment $R\mapsto(\D(R),\L_{R/A})$ is a big \'etale sheaf over $A$ of pointed categories, and the above definition for $\L_{Z/A}$ is just Remark \ref{valsta} expanded, so $\L_{-/A}$ takes colimits to limits of pushforwards, and so do $\L_{-/Y}$ and $\L_{X\times_Y-/-}$. In particular, to compute $\L_{Z/A}$, one can write $Z$ as a colimit of affines and take the limit just along these affines, rather than take the whole limit as in the definition. Similarly, one can replace the limit in the definition of $\L_{X/Y}$ by a smaller one, as long as the colimit of the affines is $Y$. 
\end{remark}

Let $A$ be a ring and $R=\bigoplus_{n\in\Z}R_n$ be a graded $A$-algebra, considered as acted by $\G_m$ naturally by the grading. In the following proposition, we view quasicoherent sheaves over $\spec(R)/\G_m$ as graded modules of $R$, for example by \cite[Theorem 4.1]{mou} and \cite[Proposition 2.5.1.2]{sag}.

\begin{proposition}\label{gradeddescent}
    Let $Y=\spec(R)/\G_m$ and $(Y_n)$, $g_n\colon Y_n\to Y$ be the \v{C}ech nerve of the cover $\spec(R)\to\spec(R)/\G_m$. Then the simplicial object $(g_n)_*\L_{Y_n/A}$ has its partial totalization tower $(\lim_{\Delta_{\le n}}(g_n)_*\L_{Y_n/A})_{n\in\N}$ pro-constant. 
\end{proposition}

\begin{proof}
    View stacks affine over $Y$ as graded $R$-algebras. Then $Y_n$ corresponds to $S^{(n)}=R[x_0^{\pm1},\ldots,x_n^{\pm1}]$ with $\deg(x_i)=1$, while absolutely $Y_n$ is the spectrum of its \th{$0$} grade part $S^{(n)}_0$. Since $S^{(n)}=S^{(n)}_0[x_0^{\pm1}]$, it is easy to see that for an $S^{(n)}_0$-module $M$, the pushforward along $g_n$ of $M$ is the graded module $M[x_0^{\pm1}]$. Hence there is a fiber sequence of graded modules
    $$(g_n)_*\L_{Y_n/A}=\L_{S^{(n)}_0/A}\otimes_{S^{(n)}_0}S^{(n)}\to\L_{S^{(n)}/A}\to\L_{S^{(n)}/S^{(n)}_0}=S^{(n)}.$$
    Now consider both $(\L_{S^{(n)}/A})$ and $(S^{(n)})$ as cosimplicial objects; then they are functorially made from the \v{C}ech nerve of $R\to S^{(0)}=R[x_0^{\pm1}]$; but this has a left inverse as a map of ungraded rings, so both $(\L_{S^{(n)}/A})$ and $S^{(n)}$ split as cosimplicial ungraded modules, and in particular their pro-objects of partial totalizations are pro-constant ungradedly; now this forces them to also be pro-constant gradedly, and hence so is the pro-object of partial totalizations of the fiber $((g_n)_*\L_{Y_n/A})$.
\end{proof}

\begin{proposition}\label{wt1f}
Consider the natural graded map $\L_{R/A}\to R$ defined as the left Kan extension of the map $\Omega_{R/A}\to R$, $\d r\mapsto\deg(r)r$. Then there is a natural isomorphism
$$\L_{(\spec(R)/\G_m)/A}\cong\fib(\L_{R/A}\to R)$$
of graded $R$-modules. 
\end{proposition}

\begin{proof}
Descend the fiber sequence in the proof of Proposition \ref{gradeddescent} and we have
$$\L_{(\spec(R)/\G_m)/A}=\fib(\L_{R/A}\to\L_{R/(\spec(R)/\G_m)})$$
as graded $R$-modules. Now view stacks affine over $\spec(R)/\G_m$ as graded rings over $R$. Then by definition, the stack $\spec(R)$ over $\spec(R)/\G_m$ corresponds to the ring $R[t^{\pm1}]$ with $\deg(t)=1$. Note that here $R=\RG(\spec(R),\O_{\spec(R)})$ is identified with the \th{$0$} grade part of $R[t^{\pm1}]$, which is $\bigoplus_{n\in\Z}R_nt^n$, while on the other hand $\L_{R/(\spec(R)/\G_m)}=\RG(\spec(R),\L_{R/(\spec(R)/\G_m)})$ is identified with that of $\L_{R[t^{\pm1}]/R}=R[t^{\pm1}]\d t$, which is $\bigoplus_{n\in\Z}R_nt^{n-1}\d t$. Therefore, the map $\L_{R/A}\to\L_{R/(\spec(R)/\G_m)}$ sends $r\in R_n$ to $\d(rt^n)=nrt^{n-1}\d t$, which gives the identification of $\L_{(\spec(R)/\G_m)/A}$ in the proposition. 
\end{proof}

The following corollary is \cite[Proposition B.9]{apc}, which was left as an exercise by the authors there. 

\begin{corollary}\label{grh}
If $R$ is $\N$-graded, then $\L_{(\spec(R)/\G_m)/A}$ is also $\N$-graded, with \th{$0$} grade part $R_0[-1]\oplus\L_{R_0/A}$, which is the same as $\L_{(\spec(R_0)/\G_m)/A}$. Therefore, for every integer $n$, the natural map
$$\RG(\spec(R_0)/\G_m,\LO_{(\spec(R_0)/\G_m)/A}^n)\to\RG(\spec(R)/\G_m,\LO_{(\spec(R)/\G_m)/A}^n)$$
is an isomorphism. Namely, the Hodge cohomology satisfies strong weighted homotopy invariance. 
\end{corollary}

\begin{proof}
Clearly, if $R$ is $\N$-graded, then so is $\L_{R/A}$, with \th{$0$} graded part $\L_{R_0/A}$, and thus so is $\L_{(\spec(R)/\G_m)/A}=\fib(\L_{R/A}\to R)$, with \th{$0$} graded part $\fib(0\colon\L_{R_0/A}\to R_0)=R_0[-1]\oplus\L_{R_0/A}$. Note that this only depends on $R_0$, so if we replace $R$ by another $\N$-graded $A$-algebra with the same $R_0$, the cotangent complex will be the same. This applies in particular to $R_0$, as a graded ring purely in the \th{$0$} grade. Now by Proposition \ref{gradeddescent}, one can commute the wedge power with the \v{C}ech descent and conclude that 
$$\LO_{(\spec(R)/\G_m)/A}^n=\LL^n\L_{(\spec(R)/\G_m)/A},$$
where the wedge power is over $\spec(R)/\G_m$. So the desired isomorphism follows from the fact that the wedge power of an $\N$-graded $R$-module $M$ is still $\N$-graded, whose \th{$0$} graded part is just the wedge power of the $R_0$-module $M_0$. 
\end{proof}

\begin{example}\label{bgh}
Take $A=R=\Z$. We get
$$\L_{\B\G_m/\Z}=\fib(0\to\Z)=\Z[-1].$$
\end{example}

\begin{example}\label{a1h}
Take $A=\Z$ and $R=\Z[t]$ with $\deg(t)=1$. We get
$$\L_{(\A^1/\G_m)/\Z}=\fib(\Z[t]\d t\to\Z[t])=\Z[-1]$$
as a graded $\Z[t]$-module purely in \th{$0$} grade, since the map here is $\d t\mapsto t$.
\end{example}

Next we discuss the \'etale cohomology. 

\begin{lemma}\label{eta1}
Let $S$ be a $\Z[1/p]$-scheme, $\pi\colon L\to S$ be a line bundle, and $i_0\colon S\to L$ be the zero section. Let $\M\in\D_{\et}(S,\Z_p)$. Then $\pi$ or equivalently $i_0$ induces an isomorphism
$$\RG_\et(S,\M)\cong\RG_\et(L,\pi^*\M).$$
\end{lemma}

\begin{proof}
Let $P=\P(L\oplus\O_S)$ be the compactification of $L$ and $\pi_P\colon P\to S$ be the projection. Let $j_\infty\colon L\to P$ and $i_\infty\colon S\to P$ be the inclusions of $L$ and $\infty$, respectively. Then there is a natural fiber sequence
$${i_\infty}_!i_\infty^!\pi_P^*\M\to\pi_P^*\M\to{j_\infty}_*\pi^*\M.$$
By the \'etale Poincar\'e duality, $\pi_P^*\M=\pi_P^!\M(-1)[-2]$, so $i_\infty^!\pi_P^*\M=\M(-1)[-2]$. Now take $\RG_\et(P,-)$ of the above sequence, and the lemma follows from the classical computation of the \'etale cohomology of the projective line. 
\end{proof}

Let $R=\bigoplus_{i\in\N}R_i$ be a graded $\Z[1/p]$-algebra, considered as acted by $\G_m$ naturally by the grading. Let $\pi\colon\spec(R)\to\spec(R_0)$ and $i\colon\spec(R_0)\to\spec(R)$ denote the obvious maps. 

\begin{proposition}\label{etgr}
Let $\M\in\D_\et(\spec(R_0),\Z_p)$. Then $\pi$ or equivalently $i$ induces an isomorphism
$$\RG_\et(\spec(R_0),\M)\cong\RG_\et(\spec(R),\pi^*\M).$$
\end{proposition}

\begin{proof}
Since the \'etale cohomology is insensitive to thickenings, without loss of generality we can assume that $R$ is classical. Let $U=\spec(R)\setminus\spec(R_0)$ and let $\pi_U\colon U\to\spec(R_0)$ and $j_U\colon U\to\spec(R)$ denote the obvious maps. By the exact triangle
$${j_U}_!\pi_U^*\M\to\pi^*\M\to i_*\M,$$
it suffices to prove $\RG_\et(\spec(R),{j_U}_!\pi_U^*\M)=0$. Let $X$ be the blowup of $\spec(R_0)$ in $\spec(R)$, and let $D$ denote the exceptional divisor. Then $U$ is also $X\setminus D$, and it suffices to prove $\RG_\et(X,{j_U}_!\pi_U^*\M)=0$, where $j_U$ now denotes the open immersion $U\to X$. Let $\O(-1)$ denote the invertible sheaf on $\proj(R)$ as usual. Then by the construction of blowup, $X$ is the total space of $\O(-1)$ and $D$ is $\proj(R)$ itself, so the proposition follows from Lemma \ref{eta1}. 
\end{proof}

\begin{corollary}\label{etwt}
For every $n\in\Z$, 
$$\RG_\et(\spec(R_0)/\G_m,\Z_p(n))\cong\RG_\et(\spec(R)/\G_m,\Z_p(n)).$$
Namely, the \'etale cohomology satisfies strong weighted homotopy invariance. 
\end{corollary}

\begin{proof}
This follows immediately from Proposition \ref{etgr} by descent. 
\end{proof}

Finally we discuss the \'etale cohomology of the adic generic fiber, generalizing \cite[Remark 5.3.7]{prismartin}. Therefore let $\FF(R)=\RG_\et(\spec(\hat{R}[1/p]),\Z_p(n))$ as in Example \ref{etadic}, where $n\in\Z$ is fixed. Recall from \cite[\S 9]{bs19} that, for $R$ over a perfectoid, 
\begin{equation}\label{etcomp}
    \RG_\et(\spec(\hat{R}[1/p]),\Z_p)=(\Prism_R[1/I]_p^\wedge)^{\varphi=1}=(\Prism_{R,\perf}[1/I]_p^\wedge)^{\varphi=1}.
\end{equation}
In what follows we will use this to control $\FF$. Before this, let us extend the functor $\Prism_{-,\perf}$ to arbitrary rings. 

\begin{definition}[Perfect prismatic cohomology]
Recall from \cite[Lemma 8.8]{bs19} that perfectoid rings form a basis of the topology $\arc_p$ on $\Ring$. For $m\in\Z$, let $\Prism^{[m]}_{-,\perf}$ denote the $\arc_p$ sheaves on $\Ring$ whose values on a perfectoid ring $R$ is $I^m\Prism_R$. Obviously, they are always coconnective and depend on the $\pi_0$ of a ring. By \cite[Proposition 8.10, Corollary 8.11]{bs19}, $\Prism_{-,\perf}=\Prism^{[0]}_{-,\perf}$ is well-defined and coincides with the notion there on rings over perfectoids. Let $\Prism_{-,\perf}[1/I]=\colim_m\Prism^{[m]}_{-,\perf}$ as a functor. By Lemma \ref{colimlimcoco}, it is automatically an $\arc_p$ sheaf, and we also have
$$\RG_\et(\spec(\hat{R}[1/p]),\Z_p)=(\Prism_{R,\perf}[1/I]_p^\wedge)^{\varphi=1}$$
for any ring $R$. 
\end{definition}

Now let $R=\bigoplus_{i\in\N}R_i$ be a graded ring, considered as acted by $\G_m$ naturally by the grading. 

\begin{proposition}\label{perfwt}
The perfect prismatic cohomology satisfies strong weighted homotopy invariance. Namely,
$$\Prism_{\spec(R)/\G_m,\perf}=\Prism_{\spec(R_0)/\G_m,\perf}.$$
\end{proposition}

\begin{proof}
Since $\spec(R)/\G_m$ and $\spec(R_0)/\G_m$ are stacks over $R_0$, both sides are $\arc_p$ sheaves with respect to $R_0$, so we can assume that $R_0$ is perfectoid. Consider the \v{C}ech nerves of $\spec(R)\to\spec(R)/\G_m$ and $\spec(R_0)\to\spec(R_0)/\G_m$, from which we know that $\Prism_{\spec(R)/\G_m,\perf}=\lim_{i\in\Delta}\Prism_{\spec(R)\times\G_m^i,\perf}$ and similarly for $\Prism_{\spec(R_0)/\G_m,\perf}$. Therefore: 
\begin{itemize}
    \item We can assume that $R$ is classical. 
    \item By Lemma \ref{colimlimcoco}, we can pass to limit and assume that $R$ is finitely presented over $R_0$. Say $R=\pi_0(R_0[x_1,\ldots,x_a]/(f_1,\ldots,f_b))$, where $x_1,\ldots,x_a$ and $f_1,\ldots,f_b$ are homogeneous in positive grade. 
    \item We can replace $R$ by the non-classical $R_0[x_1,\ldots,x_a]/(f_1,\ldots,f_b)$. 
\end{itemize}
Now we are in the case that $R_0$ and $R$ are homotopically bounded $p$-quasisyntomic rings: $R_0$ is perfectoid and hence $p$-quasisyntomic, and $R$ is a derived complete intersection over $R_0$. Therefore by Lemma \ref{qsynlimiso}, we have $\Prism_{\spec(R)/\G_m}=\Prism_{\spec(R_0)/\G_m}$. Reducing modulo $p$, since the Frobenius is functorially nullhomotopic in positive homotopical degrees (see the proof of \cite[Lemma 8.4]{bs19}, or use \cite[Remark 11.8]{witt}, since $\Prism_-$ actually has derived rings as values), by Lemma \ref{colimlimcoco} we have $\Prism_{\spec(R)/\G_m,\perf}/p=\Prism_{\spec(R_0)/\G_m,\perf}/p$. Finally by $p$-completeness of both sides we get the desired equality. 
\end{proof}

\begin{remark}
By similar limiting and replacement argument, then using Remark \ref{qsyndescent} to descend from something over a perfectoid, one can prove in steps that:
\begin{itemize}
    \item The functors $\Prism_{-,\perf}$ commutes with filtered colimits of rings. 
    \item For any ring $S$ we have $\Prism_{S,\perf}=\colim_\varphi\Prism_S$, where the colimit is taken in the category of $\Prism_S$-modules that are complete along $\Prism_S\to\overline{\Prism}_S/p$.
    \item Equation (\ref{etcomp}) holds for rings not necessarily over a perfectoid. 
\end{itemize}
\end{remark}

\begin{corollary}\label{adicwt}
The \'etale cohomology of the adic generic fiber satisfies strong weighted homotopy invariance. Namely,
$$\FF(\spec(R)/\G_m)=\FF(\spec(R_0)/\G_m).$$
\end{corollary}

\begin{proof}
By \cite[Corollary 6.17]{arc}, both sides are $\arc_p$ sheaves with respect to $R_0$. Therefore we can assume that $R_0$ is over $\Z_p^\cyc$. Since $\Z_p(n)\cong\Z_p$ on $\spec(\Q_p^\cyc)$, we can assume $n=0$. Again use \v{C}ech nerves to compute both sides. Since $\RG_\et(\spec(\hat{S}[1/p]),\Z_p)=(\Prism_{S,\perf}[1/I]_p^\wedge)^{\varphi=1}$ and $\Prism_{S,\perf}$ is coconnective for any ring $S$, the desired equality follows from Proposition \ref{perfwt} and Lemma \ref{colimlimcoco}. 
\end{proof}

\begin{theorem}\label{synwt}
The syntomic cohomology satisfies strong weighted homotopy invariance. Namely, for all $n\in\Z$,
$$\RG_\syn(\spec(R)/\G_m,\Z_p(n))=\RG_\syn(\spec(R_0)/\G_m,\Z_p(n)).$$
\end{theorem}

\begin{proof}
Combine Lemma \ref{limiso}, Corollary \ref{grh}, Corollary \ref{etwt}, and Corollary \ref{adicwt}. 
\end{proof}

\begin{remark}
The failure of strong weighted homotopy invariance of the absolute prismatic cohomology $\Prism_-$ results from the failure of Lemma \ref{qsynlimiso} for general stacks. Probably, this is because we are using the wrong category of values, and may be remedied by viewing $\Prism_-$ as taking values in the category of ``absolute $F$-gauges'', along with a careful study of the strong weighted homotopy invariance of the (non-Hodge-completed) derived de Rham cohomology.
\end{remark}

\section{Thom classes}

We now produce the theory of Thom classes for cohomologies that satisfy a coconnectivity condition. In this and the next sections we will often encounter relative cohomology, so for a big \'etale sheaf $\FF$ with values in a pointed category $\CC$: 
\begin{itemize}
    \item For stacks $X$ and $Y$ with a map $Y\to X$ understood, we let $\FF(X,Y)$ denote $\fib(\FF(X)\to\FF(Y))$.
    \item For a pointed stack $*\to X$, we let $\tilde\FF(X)$ denote $\FF(X,*)$. 
\end{itemize}

Throughout this section we are in the following situation:

\begin{situation}\label{chernsit}
Let $\FF=\bigoplus_{n\in\Z}\FF^n$ be a big \'etale sheaf with values in $\Gr\CAlg(\Sp)$. Fix a point $c_1\in\Omega^\infty\FF^1(\B\G_m)$, call it the first Chern class, and suppose that:
\begin{enumerate}
    \item $\FF$ satisfies the projective bundle formula. More precisely, for any $r\in\N$ and any rank $r$ vector bundle $V\to X$, $(c_1(\O(1))^i)_{i=0}^{r-1}$ induces an isomorphism
    $$\bigoplus_{i=0}^{r-1}\FF^{\star-i}(X)\cong\FF^\star(\P(V)).$$
    \item $\FF$ satisfies the strong weighted homotopy invariance of Definition \ref{wtinv}. 
    \item\label{cocoassumption} $\FF^{\le0}(\spec(\Z))$ is coconnective. 
\end{enumerate}
\end{situation}

Note that with the projective bundle formula and the strong weighted homotopy invariance in hand, the proof of \cite[Theorem 9.3.1]{apc} actually shows that: 
\begin{itemize}
    \item $\FF$ has Chern classes, i.e.\ for $r\in\N$ and a stack $X$ with $\FF^{\le0}(X)$ coconnective, 
    $$\pi_*\FF(X\times\B\GL_r)=(\pi_*\FF(X))[c_1,\ldots,c_r]$$
    as doubly graded rings, where $c_i\in\pi_0\FF^i(\B\GL_r)$. These Chern classes are additive, i.e.\ for $r,s\in\N$, the direct sum map $\B\GL_r\times\B\GL_s\to\B\GL_{r+s}$ induces the map $c_i\mapsto\sum_{j=0}^ic'_jc''_{i-j}$ on homotopy groups. 
\end{itemize}

The prototypical examples are $\FF^n(R)$ being $\LO_R^n[n]$, $\RG_\et(R[1/p],\Z_p(n))[2n]$, and $\RG_\syn(R,\Z_p(n))[2n]$, by results of \cite[\S 9.1]{apc} and the previous section. The main result of this section can be summarized as follows: 

\begin{theorem}\label{thom}
To every stack $X$ and every rank $r$ vector bundle $V\to X$, one can assign the Thom class
$$\Th_V^\FF\in\Omega^\infty\FF^r(V,V\setminus0),$$
where $0$ denotes the zero section of $V$. The assignment is uniquely determined (up to contractible ambiguity) by the following requirements: 
\begin{description}
    \item[Functoriality] For $f\colon Y\to X$ a morphism of stacks, we have $\Th_{f^*V}^\FF=f^*\Th_V^\FF$. Rigorously speaking, let $\Bun_r$ denote the category
    $$\{V\to X\mid\text{$X$ is a stack, $V$ is a rank $r$ vector bundle over $X$}\}$$
    with morphisms the pullback squares. Then $\Th^\FF$ is a natural transformation $*\to\Omega^\infty\FF^r(V,V\setminus0)$ in the functor category $\Fun(\Bun_r^\op,\Ani)$.
    \item[Normalization] After restricting along the zero section $0\colon X\to V$ and taking $\pi_0$, $\Th_V^\FF$ becomes the Chern class $c_r(V)$. 
\end{description}
Moreover, the Thom classes are natural in $\FF$ and have the following properties:
\begin{description}
    \item[Thom Isomorphism] Promote $V$ to the vector bundle $V/\G_m\to X/\G_m=X\times\B\G_m$ by scalar multiplication. Then $\Th^\FF_{V/\G_m}$ induces an isomorphism
    $$\FF^\star(X\times\B\G_m)\cong\FF^{\star+r}(V/\G_m,\P(V)).$$
    \item[Additivity] If $V=V'\oplus V''$, then $\Th_V^\FF=\Th_{V'}^\FF\boxtimes\Th_{V''}^\FF$, i.e.\ the product of the pullback of $\Th_{V'}$ in $\Omega^\infty\FF^{r'}(V,V\setminus V'')$ and the pullback of $\Th_{V''}$ in $\Omega^\infty\FF^{r''}(V,V\setminus V')$, where $r'=\rk(V')$, $r''=\rk(V'')$.
\end{description}
\end{theorem}

Note that the category $\Bun_r$ has a final object, i.e.\ the universal rank $r$ vector bundle which we denote by $\E_r\to\B\GL_r$. So by the Yoneda Lemma, in order to prove Theorem \ref{thom}, we only need to treat the universal case, i.e.\ $X=\B\GL_r$ and $V=\E_r$. Therefore let us study the cohomology of $\E_r$ and $\E_r\setminus0$. Recall that by definition $\E_r=\G_a^r/\GL_r$, where $\GL_r$ acts on $\G_a^r$ tautologically. 

\begin{proposition}\label{ba}
As stacks over $\B\GL_r$ we have $\E_r\setminus0\cong\B\aff_{r-1}$ where
$$\aff_{r-1}=\G_a^{r-1}\rtimes\GL_{r-1}=\begin{bmatrix}
1&\G_a^{r-1}\\
0&\GL_{r-1}
\end{bmatrix}\subseteq\GL_r.$$
\end{proposition}

\begin{proof}
This is simply because $\E_r\setminus0=(\G_a^r\setminus0)/\GL_r$, while $\GL_r$ acts transitively on $\G_a^r\setminus0$ with point stabilizer $\aff_{r-1}$.
\end{proof}

\begin{proposition}\label{isyn}
Let $X$ be a stack and $G\subseteq\GL_{r,X}$ be a subgroup smooth over $X$ that contains the central $\G_{m,X}\to\GL_{r,X}$. Then the projection $\G_{a,X}^r/G\to\B G$ induces an isomorphism $\FF(\B G)\cong\FF(\G_{a,X}^r/G)$. In particular, taking $G=\GL_{r,X}$ we have that the projection $X\times\E_r\to X\times\B\GL_r$ induces an isomorphism $\FF(X\times\B\GL_r)\cong\FF(X\times\E_r)$.
\end{proposition}

\begin{proof}
The proposition basically follows from the proof of \cite[Corollary 9.2.10]{apc}, but the statement there does not contain the case here, so we redo the proof. 

Let everything be over $X$. Let $H=G/\G_m$. Then $H$ is also smooth, and
$$\begin{tikzcd}
\B\G_m\ar[r]\ar[d]&\B G\ar[d]\\
*\ar[r]&\B H
\end{tikzcd}$$
is a fiber square. View $\B G$ as over $\B H$. If we can prove that for every stack $Y$ over $X$ and every $X$-map $Y\to\B H$, the base change of $\G_a^r/G\to*/G$ to $Y$ induces an isomorphism after taking $\FF$, then the proposition will follow by descent. Now the base change is nothing but $\G_{a,Y}^r/\G_{m,Y}\to Y/\G_{m,Y}$, where $\G_{m,Y}$ acts on $\G_{a,Y}^r$ by scalar multiplication, so we are done by weighted homotopy invariance. 
\end{proof}

\begin{proposition}\label{coha}
For any stack $X$, the natural map $\GL_{r-1}\to\aff_{r-1}$ induces an isomorphism 
$$\FF(X\times\B\aff_{r-1})\cong\FF(X\times\B\GL_{r-1}).$$
\end{proposition}

\begin{proof}
Let everything be over $X$ to simplify notations. By $\aff_{r-1}=\G_a^{r-1}\rtimes\GL_{r-1}$ we have $\B\aff_{r-1}=\B\G_a^{r-1}/\GL_{r-1}$, where $\GL_{r-1}$ tautologically acts on $\G_a^{r-1}$. Under this identification, $\B\GL_{r-1}\to\B\aff_{r-1}$ corresponds to the map $*/\GL_{r-1}\to\B\G_a^{r-1}/\GL_{r-1}$ defined as the canonical covering $*\to\B\G_a^{r-1}$ quotient by $\GL_{r-1}$. By descent, we only need to prove that the \v{C}ech nerves of $*/\GL_{r-1}\to\B\G_a^{r-1}/\GL_{r-1}$ and $*/\GL_{r-1}\to*/\GL_{r-1}$ have the same cohomology on each term, i.e.\ for each $k\in\N$ the map $(\G_a^{r-1})^k/\GL_{r-1}\to*/\GL_{r-1}$ induces an isomorphism after applying $\FF$, but this follows immediately from Proposition \ref{isyn}.
\end{proof}

\begin{proposition}\label{cr0}
For any stack $X$ with $\FF^{\le0}(X)$ coconnective, the natural restriction map $\pi_*\FF(X\times\E_r)\to\pi_*\FF(X\times(\E_r\setminus0))$ is the map
$$(\pi_*\FF(X))[c_1,\ldots,c_r]\to(\pi_*\FF(X))[c_1,\ldots,c_{r-1}]$$
that leaves $c_1,\ldots,c_{r-1}$ unchanged and maps $c_r$ to $0$. In particular, it is surjective, and thus
\begin{align*}
    \pi_*\FF(X\times\E_r,X\times(\E_r\setminus0))&=\ker(\pi_*\FF(X\times\E_r)\to\pi_*\FF(X\times(\E_r\setminus0)))\\
    &=c_r(\pi_*\FF(X)[c_1,\ldots,c_r]).
\end{align*}
\end{proposition}

\begin{proof}
By Proposition \ref{ba}, Proposition \ref{isyn} and Proposition \ref{coha}, the map $X\times(\E_r\setminus0)\to X\times\E_r$ is cohomologically the same as the map $X\times\B\GL_{r-1}\to X\times\B\GL_r$ induced by adding a trivial line bundle. So the proposition follows from the additivity formula of Chern classes.
\end{proof}

\begin{proof}[Proof of Theorem \ref{thom}]
By Proposition \ref{cr0},
$$\pi_*\FF(\E_r,\E_r\setminus0)=c_r(\pi_*\FF(\spec(\Z))[c_1,\ldots,c_r]).$$
Since $c_i\in\pi_0\FF^i$, among the right hand side only $\pi_*\FF^{\le0}(\spec(\Z))$ may contribute to $\pi_*\FF^r(\E_r,\E_r\setminus0)$. By assumption (\ref{cocoassumption}) of Situation \ref{chernsit}, this implies that $\FF^r(\E_r,\E_r\setminus0)$ is coconnective, so $\Omega^\infty\FF^r(\E_r,\E_r\setminus0)$ is a $0$-truncated anima, in which $c_r$ is naturally a point. Therefore by normalization, the Thom class of the universal bundle $\E_r\to\B\GL_r$ can be and must be defined as $c_r\in\Omega^\infty\FF^r(\E_r,\E_r\setminus0)$, which is unique up to contractible choice. Now the Thom class of a general rank $r$ vector bundle is uniquely determined by the universal case. 

To prove the additivity, one reduces to the universal case, where it follows from the additivity of the Chern class. To prove the Thom isomorphism,  note by weighted homotopy invariance that $V/\G_m\to\B\G_m\times X$ induces an isomorphism $\FF(\B\G_m\times X)\cong\FF(V/\G_m)$, while $(V/\G_m)\setminus(0/\G_m)=\P(V)$ is the projective bundle. Now by additivity and the projective bundle formula, the Thom isomorphism reduces to the case $r=1$, where it follows immediately by the projective bundle formula. 
\end{proof}

\begin{remark}
    It seems magical that, starting with Chern classes in the homotopy groups, we are able to get the desired Thom class up to contractible ambiguity. This is in fact because the assumption (\ref{cocoassumption}) of Situation \ref{chernsit} essentially guarantees that the ``candidate space'' $\Omega^\infty\FF^r(\E_r,\E_r\setminus0)$ of the Thom class is discrete, so once we have specified the connected component that the class lives in, we have specified the class up to contractible ambiguity. 
\end{remark}

\begin{remark}\label{unithomadd}
Alternatively, one can impose the normalization only for line bundles but add the additivity of Thom classes to get the uniqueness statement, because by additivity of Chern classes, the natural map
$$\pi_*\FF(\E_r,\E_r\setminus0)\to\pi_*\FF(\E_1^r,\E_1^r\setminus0)$$
is injective, sending $c_r$ to the box product of the $r$ different $c_1$'s on the right. 
\end{remark}

\begin{remark}[Naturality]\label{natthom}
By uniqueness, the Thom class $\Th^\FF$ is natural in $\FF$, with respect to maps that preserve $c_1$. In particular, it is natural along the map
$$\Z_p(n)\to\fil^n\Prism_-\{n\}\to\fil^n\widehat{\dr_-}\to\widehat{\LO^n_-}[-n]$$
from the syntomic cohomology to the $p$-completed Hodge cohomology, by \cite[Theorem 7.6.2]{apc}. 
\end{remark}

\begin{remark}[Identifying $c_1^\hod$ and $\Th_1^\hod$]\label{c1hod}
Take $\FF^n(R)=\LO^n_R[n]$. By Example \ref{bgh}, $\FF^1(\B\G_m)=\Z$, in which $1$ corresponds to the Hodge first Chern class originally defined by the map $\dlog\colon\G_m\to\LO^1$, thanks to the proof of Proposition \ref{wt1f}. Therefore the Thom class is also $1\in\Z$, under the identification of Example \ref{a1h}. 
\end{remark}

\begin{remark}
Our way of stating Situation \ref{chernsit} and Theorem \ref{thom} is convenient for later uses in the paper, but far from optimal: 
\begin{itemize}
    \item We can use general base stacks in place of $\spec(\Z)$. 
    \item Formulated properly, the values of $\FF$ need not be graded. See \cite[\S 3]{p1mot}.
    \item We only need the smooth weighted homotopy invariance rather than the strong one, except for the Thom isomorphism. 
    \item With $\P^1$-homotopy techniques as in \cite{p1mot}, one should be able to get rid of all the coconnectivity assumptions in this section. This can be for example applied to the Hodge cohomology over an animated ring. 
\end{itemize}
\end{remark}

\section{Cycle classes}\label{cycle}

In this section we construct cycle classes, by reducing to Thom classes with deformation to the normal cone, following \cite[Lecture 15]{dr}. In \S\ref{subsectiondnc}, we define the weighted deformation to the normal cone for a general closed immersion, and prove that for any regular immersion, the cohomology of this deformation space coincides with that of its special fiber, for several cohomology theories. These results are then used in \S\ref{subsectioncycle} to generate cycle classes from Thom classes. 

\subsection{Deformation to the normal cone}\label{subsectiondnc}
Classically, deformation to the normal cone is a construction that deforms a regular immersion into the zero section of its normal cone. Here, with derived algebraic geometry in hand, we are able to perform this construction for all closed immersions of derived stacks, and it will behave well especially for derived regular immersions. 

Before going into the construction, we first study closed immersions of derived stacks in general. 

\begin{definition}[Animated pairs]
A \emph{pair} is a map $A\to B$ of animated rings with $\pi_0A\to\pi_0B$ surjective. Let $\Pair$ denote the category of pairs, i.e.\ the full category of $\Fun(\Delta^1,\Ring)$ spanned by the pairs. We say that a pair $A\to B$ is \emph{classical} if both $A$ and $B$ are classical. This coincides with the classical notion of a pair: a ring $A$ with an ideal $I=\ker(A\to B)$ of it.  
\end{definition}

\begin{remark}[Colimits in $\Pair$]\label{coliminpair}
Let $(A_i\to B_i)_{i\in I}$ be a diagram in $\Pair$. Then its colimit in $\Fun(\Delta^1,\Ring)$ is $\colim A_i\to\colim B_i$. Since $\pi_0\colon\Ring\to\Ring^\heartsuit$ preserves colimits and colimits in $\Ring^\heartsuit$ preserve surjections, we see that $(\colim A_i\to \colim B_i)\in\Pair$, so it is also the colimit of $(A_i\to B_i)_{i\in I}$ in $\Pair$. 

Alternatively, one can take the colimit of $(A_i\to B_i)_{i\in I}$ in two steps: first take $A=\colim A_i$ and take base changes $B_i'=B_i\otimes_{A_i}A$, then take $B=\colim B_i'$ in the category of $A$-algebras, and the resulting $A\to B$ will be the desired colimit. 
\end{remark}

\begin{remark}[Projective generators of $\Pair$]\label{pgp}
Consider two objects $\Z[x]\to\Z$ and $\Z[y]\to\Z[y]$ in $\Pair$, where the maps are $x\mapsto0$ and $\id$, respectively. It is easy to see that $\Z[x]\to\Z$ corepresents the functor $(A\to B)\mapsto \fib(A\to B)$, while $\Z[y]\to\Z[y]$ corepresents the functor $(A\to B)\mapsto A$. Since both functors commute with sifted colimits, both objects are projective, and hence so are their finite products $\Z[x_1,\ldots,x_r,y_1,\ldots,y_s]\to\Z[y_1,\ldots,y_s]$ for all $r,s\in\N$. Also, since the functor $\Pair\to\Ani\times\Ani$, $(A\to B)\mapsto(A,\fib(A\to B))$ is conservative, the two objects generate $\Pair$. Let
$$\Pair_\cc=\{\Z[x_1,\ldots,x_r,y_1,\ldots,y_s]\to\Z[y_1,\ldots,y_s]\mid r,s\in\N\}$$
be the full subcategory of $\Pair$ spanned by these finite products. Then by \cite[Proposition 5.5.8.22]{htt}, $\Pair=\PS(\Pair_\cc)$, namely the animation of $\Pair_\cc$. 
\end{remark}

\begin{remark}[$n$-truncated objects in $\Pair$]
By Remark \ref{pgp} and \cite[Remark 5.5.8.26]{htt}, for $n\in\N$, an object $(A\to B)\in\Pair$ is $n$-truncated if and only if both $A$ and $\fib(A\to B)$ are $n$-truncated as animas. Therefore, a $0$-truncated object in $\Pair$ is not always classical. See also \cite[Remark 3.9]{mao}. 
\end{remark}

\begin{definition}[Closed immersions]
Let $Y\to X$ be a map of stacks. We say that it is a \emph{closed immersion} if it is corepresentable by pairs, or in other words, for any ring $A$ and any map $\spec(A)\to X$, the pullback $Y\times_X\spec(A)$ is some $\spec(B)$ with $\pi_0A\twoheadrightarrow\pi_0B$. In this case we also say that $Y$ is a \emph{closed substack} of $X$. 
\end{definition}

\begin{definition}[Regular immersions]
We say that a pair $A\to B$ is \emph{regular of codimension $r$} if Zariski locally it is a base change of the pair $\Z[x_1,\ldots,x_r]\to\Z$, $x_1,\ldots,x_r\mapsto0$. We say that a closed immersion $Y\to X$ of stacks is a \emph{regular immersion of codimension $r$} if it is corepresentable by regular pairs of codimension $r$; we say that it is a \emph{regular immersion} if it is a disjoint union of regular immersions of codimension $r$ for different $r\in\N$. 
\end{definition}

\begin{definition}[The classifying stack of pairs]\label{pair}
Consider the functor $\Pair\to\Ring$, $(A\to B)\mapsto A$. It is a cocartesian fibration with small fibers, whose cocartesian edges are
$$\{(A\to B)\to(A'\to B')\mid B'=B\otimes_AA'\}.$$
Therefore, by the Grothendieck construction in \cite[\S 3.2]{htt}, it corresponds to a functor $\Ring\to\Cat$, taking $A$ to the category of all pairs $A\to B$. 

Let $\pair\colon\Ring\to\Ani$ be the functor taking $A$ to the anima of pairs $A\to B$, corresponding under the Grothendieck construction to the subcategory of $\Pair$ that contains only the cocartesian edges. Then it is obviously a stack, which we call \emph{the classifying stack of pairs} or \emph{the classifying stack of closed immersions}. Denote the universal closed immersion by $\closed\to\pair$. Then by definition,
$$\closed(A)=\{\text{a pair $A\to B$ along with a section $B\to A$}\},$$
i.e.\ a closed immersion $\spec(B)\to\spec(A)$ with a point of $\spec(B)$. Denote the universal open complement by $\open\to\pair$. Then similarly
$$\open(A)=\{\text{a pair $A\to B$ along with a section $\spec(A)\to\spec(A)\setminus\spec(B)$}\},$$
which implies that $\open=*$, since the section forces $B=0$. 

For $r\in\N$, we can also define the classifying stacks of regular immersions of codimension $r$ in exactly the same way, which we denote by $\pair_r$, and similarly $\closed_r$ and $\open_r$ for the universal closed and open substacks. By definition, $\pair_r$ is the image of the \'etale sheaf map $\A^r\to\pair$ classifying the standard immersion $*\to\A^r$. Note that $\open_r$ is always $*$, except for $r=0$ where it is empty. 

The tensor product of $A$-algebras gives $\pair(A)$ an $\mathbb{E}_\infty$-monoid structure, making $\pair$ an $\mathbb{E}_\infty$-monoid stack. Note that the point $*=\open\to\pair$ is a zero element of this monoid; in other words, $\pair$ is an $\mathbb{E}_\infty$-monoid in the category of pointed stacks, with respect to the smash product. 
\end{definition}

\begin{remark}
By base change of the cotangent complex, for $Y\to X$ a regular immersion (of codimension $r$), it is easy to see that $\L_{Y/X}[-1]$ is locally free (of rank $r$) on $Y$. 
\end{remark}

Now we construct the derived Rees algebra. The following proposition is \cite[Proposition 13.3]{cplx}.

\begin{proposition}[cf.\ {\cite[Corollary 3.54]{mao}}]\label{rees}
Let $(A\to B)\in\Pair$. Consider the initial animated filtered $A$-algebra $I^\bullet$ equipped with a map of animated $A$-algebras $B\to I^0/I^1$. Then we have $I^n=A$ for $n\le0$, and $I^0/I^1=B$; in particular $I=I^1$ is the fiber of $A\to B$. Moreover, we have $I/I^2=\L_{B/A}[-1]$ and $I^n/I^{n+1}=\lsym_B^n(I/I^2)$ for $n\ge0$. The formation of $I^\bullet$ preserves all colimits as a functor $\Pair\to\Fil\Ring$, commutes with base change in $A$, and if $A\to B$ is classical and is a local complete intersection, then $I^\bullet$ is classical and agrees with the usual filtered $A$-algebra of powers of $I$. 
\end{proposition}

\begin{remark}
Proposition \ref{rees} can be generalized to general morphisms in $\Ring$, not necessarily surjective on $\pi_0$, but one needs to use derived rings as in \cite[Example 4.3.1, Construction 4.3.4]{rak20}, in place of animated rings. It is similar to but different from \cite[\S 5.3]{rak20}. 
\end{remark}

In the following definitions, the notation $I^n$ is as in Proposition \ref{rees}. 

\begin{definition}[Rees algebra and deformation to the normal cone]\label{dtnc}
For a pair $A\to B$, its \emph{Rees algebra} is the graded $A[t]$-algebra
$$\R_{B/A}=\bigoplus_{n\in\Z}I^{-n}t^n,$$
where $t$ is called the \emph{Rees parameter}. There is an obvious graded map $\R_{B/A}\to B[t]$, which we denote by $\R(A\to B)\in\Pair$. By Proposition \ref{rees}, $\R$ commutes with all colimits, when viewed as a functor from pairs to graded pairs over the graded ring $\Z[t]$. The \emph{(weighted) deformation to the normal cone} of $A\to B$ is the closed immersion of Artin stacks obtained by quotienting $\R(A\to B)$ by the $\G_m$-action defined by the grading, i.e.\ the map
$$\spec(B)\times\A^1/\G_m=\spec(B[t])/\G_m\to\spec(\R_{B/A})/\G_m,$$
viewed as a map of stacks over $\A^1/\G_m$. Its open fiber (the fiber over $\G_m/\G_m\to\A^1/\G_m$) is $\spec(B)\to\spec(A)$, while its closed fiber (the fiber over $0/\G_m\to\A^1/\G_m$) is $\spec(B)/\G_m\to\spec(\lsym_B^\star(I/I^2))/\G_m$, which is the quotient by the $\G_m$-multiplication on the normal cone of $\spec(B)\to\spec(A)$.

This construction clearly satisfies \'etale descent and commutes with base change of $A$, and hence can be globalized and classified. More precisely, $(A\to B)\mapsto\spec(\R_{B/A})/\G_m$ defines a stack over $\pair\times\A^1/\G_m$, denoted by $\DNC$, whose fiber over $\closed\times\A^1/\G_m$ is $\closed\times\A^1/\G_m$ itself. Its restriction to $\pair\times\G_m/\G_m$ is $\id_\pair$, while its restriction to $\pair\times0/\G_m$ is the quotient by $\G_m$ of the normal cone $\underline{\spec}_\pair(\lsym^\star(\I_\closed/\I_\closed^2))$ of $\closed\to\pair$. For a closed immersion $Y\to X$ of stacks, its deformation to the normal cone can be constructed as the base change of $\closed\times\A^1/\G_m\to\DNC$ along the classifying map $X\to\pair$. 
\end{definition}

\begin{definition}[Animated blowup]
The \emph{blowup} of a pair $A\to B$ is the $A$-scheme $\bl_B(A)=\proj\left(\bigoplus_{n\in\N}I^n\right)$. The natural ``inclusion'' map $\bigoplus_{n\in\N}I^{n+1}\to\bigoplus_{n\in\N}I^n$ defines a section of the line bundle $\O(-1)$, whose zero locus is by definition $\proj\left(\bigoplus_{n\in\N}I^n/I^{n+1}\right)$, which we call the \emph{exceptional divisor}. Note that it is the projectivized normal cone $\proj(\lsym_B^\star(I/I^2))$ of $\spec(B)$ in $\spec(A)$. Similarly, this construction can also be globalized and classified. For a closed immersion $Y\to X$ of stacks, we denote its blowup by $\bl_Y(X)$. 
\end{definition}

\begin{remark}\label{blowupanddnc}
The two concepts are closely related. Namely, one can check that:
\begin{itemize}
    \item $\bl_B(A)=(\spec(\R_{B/A})\setminus(\spec(B[t])\cup\spec(A)))/\G_m$.
    \item $\spec(\R_{B/A})=\bl_B(A[t])\setminus\bl_B(A)$, where the map $A[t]\to A$ sends $t$ to $0$.
\end{itemize}
\end{remark}

Next we study the cohomology of the deformation to the normal cone. Fix an $r\in\N$, let $Y\to X$ be a regular immersion of stacks of codimension $r$, and let $\Y\to\X$ be its deformation to the normal cone, with closed fiber $\Y_0\to\X_0$. 

We begin with the Hodge cohomology. Recall that stacks affine over $\B\G_m$ correspond to graded rings, whose quasicoherent sheaves correspond to graded modules, and taking global section corresponds to taking the \th{$0$} grade. 

\begin{proposition}\label{cpt}
Let $A\to B$ be a regular pair of codimension $r$. For any $n\in\Z$, the local cohomology $\RG_{\spec(B[t])}(\LO^n_{\R_{B/A}/\Z[t]})$ and $\RG_{\spec(B)}(\LO^n_{\lsym_B^\star(I/I^2)/\Z})$ are both nonnegatively graded, and the base change map along $A[t]\to A$, $t\mapsto0$ identifies the \th{$0$} grade parts of them. 
\end{proposition}

\begin{proof}
Localizing, we can assume that $A\to B$ is a base change of $\Z[x_1,\ldots,x_r]\to\Z$. Consider the direct sums
$$\bigoplus_{n\in\Z}\RG_{\spec(B[t])}(\LO^n_{\R_{B/A}/\Z[t]}[-n])\in\D(\Z[t])$$
and
$$\bigoplus_{n\in\Z}\RG_{\spec(B)}(\LO^n_{\lsym_B^\star(I/I^2)/\Z}[-n])\in\D(\Z)$$
instead of individual complexes. Then they commute with sifted colimits in $A$ as $\Z[x_1,\ldots,x_r]$-algebras, and take coproducts of pairs to tensor products, so one reduces to the case $r$ is either $1$ or $0$ and $A\to B$ is either $\Z[x]\to\Z$ or $\Z[y]\to\Z[y]$, respectively. Now $\LO^n$ is nonzero only when $n=0$ or $n=1$. 

In the $\Z[x]\to\Z$ case, $\R_{B/A}=\Z[xt^{-1},t]$ with the surjection to $B[t]=\Z[t]$ killing $xt^{-1}$, so one computes
$$\RG_{\spec(B[t])}(\R_{B/A})=\RG_{xt^{-1}}(\Z[xt^{-1},t])=\bigoplus_{k=1}^\infty\Z[t]t^kx^{-k}[-1],$$
$$\RG_{\spec(B)}(\lsym_B^\star(I/I^2))=\RG_{xt^{-1}}(\Z[xt^{-1}])=\bigoplus_{k=1}^\infty\Z t^kx^{-k}[-1],$$
$$\RG_{\spec(B[t])}(\L_{\R_{B/A}/\Z[t]})=\RG_{xt^{-1}}(\Z[xt^{-1},t]\d(xt^{-1}))=\bigoplus_{k=1}^\infty\Z[t]t^kx^{-k}\d(xt^{-1})[-1],$$
$$\RG_{\spec(B)}(\L_{\lsym_B^\star(I/I^2)/\Z})=\RG_{xt^{-1}}(\Z[xt^{-1}]\d(xt^{-1}))=\bigoplus_{k=1}^\infty\Z t^kx^{-k}\d(xt^{-1})[-1].$$
In the $\Z[y]\to\Z[y]$ case, $A=B$, $I=0$ and $\R_{B/A}=B[t]=\Z[y,t]$, so the local cohomology does nothing, and one computes
$$\RG_{\spec(B[t])}(\R_{B/A})=\Z[y,t],$$
$$\RG_{\spec(B)}(\lsym_B^\star(I/I^2))=\Z[y],$$
$$\RG_{\spec(B[t])}(\L_{\R_{B/A}/\Z[t]})=\Z[y,t]\d y,$$
$$\RG_{\spec(B)}(\L_{\lsym_B^\star(I/I^2)/\Z})=\Z[y]\d y.$$
In all these cases one easily verifies the proposition. 
\end{proof}

\begin{corollary}\label{0hg}
For every $n\in\N$, the inclusion $\X_0\to\X$ induces an isomorphism
$$\RG_\Y(\X,\LO^n_\X)\cong\RG_{\Y_0}(\X_0,\LO^n_{\X_0}).$$
\end{corollary}

\begin{proof}
Both sides are sheaves with respect to $X$, so one can assume that $X=\spec(A)$ is affine, and hence so is $Y=\spec(B)$. Let $f\colon\X\to\A^1/\G_m$ and $f_0\colon\X_0\to\B\G_m$ be the structure maps. Then we have the fiber sequences
$$f^*\L_{\A^1/\G_m}\to\L_\X\to\L_{\X/(\A^1/\G_m)}$$
and
$$f_0^*\L_{\B\G_m}\to\L_{\X_0}\to\L_{\X_0/\B\G_m},$$
which imply that $\LO_\X^n$ is filtered by $\LO_{\X/(\A^1/\G_m)}^i\otimes f^*\LO_{\A^1/\G_m}^j$ with $i+j=n$, and similarly for $\X_0$. Now by Example \ref{a1h}, $\RG_\Y(\LO_{\X/(\A^1/\G_m)}^i\otimes f^*\LO_{\A^1/\G_m}^j)$ is computed by taking the \th{$0$} grade part of
$$\RG_{\spec(B[t])}(\LO^n_{\R_{B/A}/\Z[t]})\otimes_{\Z[t]}\LL_{\Z[t]}^j(\Z[-1]),$$
while by Example \ref{bgh}, $\RG_{\Y_0}(\LO_{\X_0/\B\G_m}^i\otimes f_0^*\LO_{\B\G_m}^j)$ is computed by taking the \th{$0$} grade part of
$$\RG_{\spec(B)}(\LO^n_{\lsym_B^\star(I/I^2)})\otimes_\Z\LL_\Z^j(\Z[-1]),$$
and by Proposition \ref{cpt} it is not hard to see that the two are isomorphic via the natural map. 
\end{proof}

We then study the \'etale cohomology. 

\begin{proposition}\label{etdnc}
Suppose $X$ and $Y$ are over $\Z[1/p]$. Let $\tilde{\Y}\to\tilde{\X}$ be the deformation to the normal cone without quotient by $\G_m$, and $\tilde{\Y}_0\to\tilde{\X}_0$ be the fiber over $0\to\A^1$. Then for every \'etale $\Z_p$-sheaf $\M$ on $X$, the inclusion $\tilde{\X}_0\to\tilde{\X}$ induces an isomorphism on local cohomology
$$\RG_{\et,\tilde{\Y}}(\tilde{\X},\tilde{f}^*\M)\cong\RG_{\et,\tilde{\Y}_0}(\tilde{\X}_0,\tilde{f}_0^*\M).$$
Here $\tilde{f}\colon\tilde{\X}\to X$ and $\tilde{f}_0\colon\tilde{\X}_0\to\tilde{\X}\to X$ are the natural maps. 
\end{proposition}

\begin{proof}
The proposition is clearly \'etale local on $X$, so we can assume that $X=\spec(A)$ is affine, and hence so is $Y=\spec(B)$. Then $\tilde{\X}=\spec(\R_{B/A})$, $\tilde{\Y}=\spec(B[t])$, $\tilde{\X}_0=\spec(\lsym_B^\star(I/I^2))$, $\tilde{\Y}_0=\spec(B)$, as in Definition \ref{dtnc}. By excision, we can assume that the pair $A\to B$ is Henselian, i.e.\ the classical pair $\pi_0A\to\pi_0B$ is Henselian. 

Now by Remark \ref{blowupanddnc} and excision, the left hand side is
$$\RG_\et(\bl_Y(\A^1_X),\bl_Y(\A^1_X)\setminus\spec(\A^1_Y),g^*\M),$$
where $g\colon\bl_Y(\A^1_X)\to X$ is the projection. To compute it, we consider the compactification $j\colon\bl_Y(\A^1_X)\to\bl_Y(\P^1_X)$, which also compactifies $\A^1_Y$ to $\P^1_Y$. Let $\pi$ denote the projections $\bl_Y(\P^1_X)\to X$, and $i$ denote the complement of $j$. Since the inclusion $Y\to\A^1_X$ has image in $0\to\A^1_X$, the blowup changes nothing near $\infty\to\P^1_X$. Therefore by the \'etale Poincar\'e duality and by excision again, we have $i^!\pi^*\M=\M(-1)[-2]$. Hence
\begin{align*}
    \RG_\et(\bl_Y(\A^1_X),g^*\M)&=\RG_\et(\bl_Y(\P^1_X),j_*g^*\M)\\
    &=\RG_\et(\bl_Y(\P^1_X),\cofib(i_!i^!\pi^*\M\to\pi^*\M))\\
    &=\cofib(\RG_\et(X,\M(-1)[-2])\to\RG_\et(\bl_Y(\P^1_X),\pi^*\M)).
\end{align*}
Note that $\bl_Y(\P^1_X)$ is proper over $X$, whose fiber over $Y$ is $\P^1_Y\sqcup_Y\P(\NN_{Y/X}\oplus\O_Y)$. Here $\NN_{Y/X}$ denotes the normal bundle and $\P$ means projectivization. Therefore by Gabber's affine proper base change \cite[\texttt{09ZI}]{stacks}, 
$$\RG_\et(\bl_Y(\P^1_X),\pi^*\M)=\RG_\et(\P^1_Y\sqcup_Y\P(\NN_{Y/X}\oplus\O_Y),\bar{\pi}^*\M),$$
where $\bar{\pi}$ denotes the corresponding projection. Take the pushout out of $\RG_\et$, and note that $\RG_\et(X,\M(-1)[-2])=\RG_\et(Y,\M(-1)[-2]|_Y)$ cancels the $\P^1_Y\sqcup_Y-$ again by the affine proper base change, we conclude that
$$\RG_\et(\bl_Y(\A^1_X),g^*\M)=\RG_\et(\P(\NN_{Y/X}\oplus\O_Y),\pi_0^*\M),$$
where $\pi_0$ denotes the projection $\P(\NN_{Y/X}\oplus\O_Y)\to X$. On the other hand, note that $\bl_Y(\A^1_X)\setminus\spec(\A^1_Y)$ is the total space of $\O(-1)$ on $\bl_Y(X)$, so by Lemma \ref{eta1} and the affine proper base change,
$$\RG_\et(\bl_Y(\A^1_X)\setminus\spec(\A^1_Y),\pi^*\M)=\RG_\et(\bl_Y(X),\beta^*\M)=\RG_\et(\P(\NN_{Y/X}),\beta_0^*\M),$$
where $\beta\colon\bl_Y(X)\to X$ and $\beta_0\colon\P(\NN_{Y/X})\to X$ are the projections. Recall that the total space of $\O(-1)$ on $\P(\NN_{Y/X})$ is $\P(\NN_{Y/X}\oplus\O_Y)\setminus0$, so we finally have
\begin{align*}
    &\quad\ \RG_\et(\bl_Y(\A^1_X),\bl_Y(\A^1_X)\setminus\spec(\A^1_Y),g^*\M)\\
    &=\fib(\RG_\et(\P(\NN_{Y/X}\oplus\O_Y),\pi_0^*\M)\to\RG_\et(\P(\NN_{Y/X}),\beta_0^*\M))\\
    &=\RG_\et(\P(\NN_{Y/X}\oplus\O_Y),\P(\NN_{Y/X}\oplus\O_Y)\setminus0,\pi_0^*\M),
\end{align*}
but this is exactly the right hand side
$$\RG_\et(\tilde{\X}_0,\tilde{\X}_0\setminus\tilde{\Y}_0,\tilde{f}_0^*\M)=\RG_\et(\NN_{Y/X},\NN_{Y/X}\setminus0,\tilde{f}_0^*\M),$$
by excision. 
\end{proof}

\begin{corollary}\label{etwtdnc}
For every $n\in\Z$, the inclusion $\X_0\to\X$ induces an isomorphism
$$\RG_{\et,\Y}(\X,\Z_p(n))\cong\RG_{\et,\Y_0}(\X_0,\Z_p(n)).$$
\end{corollary}

\begin{proof}
This follows immediately from Proposition \ref{etdnc} by descent. 
\end{proof}

We finally study the \'etale cohomology of the generic fiber. As before, we use the \'etale comparison of the prismatic cohomology to control it. 

\begin{proposition}
The inclusion $\X_0\to\X$ induces an isomorphism
$$\Prism_{\X,\X\setminus\Y,\perf}\cong\Prism_{\X_0,\X_0\setminus\Y_0,\perf}.$$
\end{proposition}

\begin{proof}
The proof is similar to that of Proposition \ref{perfwt}. Both sides are $\arc_p$ sheaves over $X$, so we can assume that $X=\spec(R)$ is the spectrum of a perfectoid ring, and then $Y=\spec(S)$ is the spectrum of an $R$-algebra. Now $S=R/(x_1,\ldots,x_r)$ is a derived complete intersection in $R$. Then $S$ is homotopically bounded and quasi-lci over $R$, so by Lemma \ref{qsynlimiso} and Corollary \ref{0hg}, we have $\Prism_{\X,\X\setminus\Y}=\Prism_{\X_0,\X_0\setminus\Y_0}$. Finally Lemma \ref{colimlimcoco} gives us the desired equality. 
\end{proof}

\begin{corollary}\label{adicdnc}
Let $\FF$ be the big \'etale sheaf $R\mapsto\RG_\et(\spec(\hat{R}[1/p]),\Z_p(n))$, where $n\in\Z$ is fixed. Then the inclusion $\X_0\to\X$ induces an isomorphism
$$\FF(\X,\X\setminus\Y)\cong\FF(\X_0,\X_0\setminus\Y_0).$$
\end{corollary}

\begin{proof}
The proof is similar to that of Corollary \ref{adicwt}. Both sides are $\arc_p$ sheaves over $X$, so we can assume that $X=\spec(R)$ for a $\Z_p^\cyc$-algebra $R$. Since $\Z_p(n)\cong\Z_p$ on $\spec(\Q_p^\cyc)$, we can assume $n=0$. Also, both are quotients of qcqs schemes by $\G_m$, so we can use \v{C}ech nerves to compute them. Since $\FF(U)=(\Prism_{U,\perf}[1/I]_p^\wedge)^{\varphi=1}$ for any qcqs scheme $U$ (which follows from the affine case by taking a finite limit), the desired equality follows from Proposition \ref{perfwt} and Lemma \ref{colimlimcoco}. 
\end{proof}

Finally, we can prove our theorem on the syntomic cohomology of the deformation to the normal cone. 

\begin{theorem}\label{syndnc}
For every integer $n$, the inclusion $\X_0\to\X$ induces an isomorphism
$$\RG_\syn(\X,\X\setminus\Y,\Z_p(n))\cong\RG_\syn(\X_0,\X_0\setminus\Y_0,\Z_p(n)).$$
\end{theorem}

\begin{proof}
Combine Lemma \ref{limiso}, Corollary \ref{0hg}, Corollary \ref{etwtdnc}, and Corollary \ref{adicdnc}. 
\end{proof}

\begin{remark}
We wonder whether the same holds for general closed immersions, not necessarily regular. For the \'etale cohomology the same proof seems to show the general statement, while for the Hodge cohomology and the \'etale cohomology of the generic fiber we have no idea. 
\end{remark}

\subsection{Constructing cycle classes}\label{subsectioncycle}
We now use what we have computed in the previous subsection to construct cycle classes from Thom classes. 

\begin{theorem}[Characteristic classes of regular pairs]\label{charclasspair}
Let $\CC$ be a pointed presentable category and let $\FF$ be a big \'etale sheaf with values in $\CC$. Fix $r\in\N$. Let $\E_r$ denotes the universal rank $r$ vector bundle over $\B\GL_r$ and $0_r$ denotes the zero section. Then since $0_r\to\E_r$ is a regular immersion of codimension $r$, it corresponds to a map $\beta_r\colon\E_r\to\pair_r$ with $\beta_r^{-1}(\closed_r)=0_r$. Assume that for any regular immersion $Y\to X$ of stacks of codimension $r$, we have
$$\FF(\X,\X\setminus\Y)=\FF(\X_0,\X_0\setminus\Y_0),$$
where $\Y\to\X$ is the deformation to the normal cone of $Y\to X$ and $\Y_0\to\X_0$ is its closed fiber. Then $\beta_r$ induces an isomorphism
$$\beta_r^*\colon\FF(\pair_r,\open_r)\cong\FF(\E_r,\E_r\setminus0_r).$$
In other words, all the $\FF$-characteristic classes of pairs that vanish on the open loci come from those of zero sections in rank $r$ vector bundles. 
\end{theorem}

\begin{proof}
Let $\DNC_r\to\pair_r\times\A^1/\G_m$ denote the deformation to the normal cone of $\pair_r$, and $\open_{\DNC_r}$ denote the complement of $\closed_r\times\A^1/\G_m$ in $\DNC_r$. Let $\DNC_{r,0}$ and $\open_{\DNC_{r,0}}$ denote the fiber of $\DNC_r$ and $\open_{\DNC_r}$, respectively. By construction, $\DNC_{r,0}$ is the weighted normal bundle of $\closed_r\to\pair_r$, so there is a natural classifying map $\DNC_{r,0}\to\E_r$. Consider the diagram
$$\begin{tikzcd}
    &\DNC_{r,0}\ar[r]\ar[d]&\E_r\ar[d]\\
    \pair_r\ar[r]&\DNC_r\ar[r]&\pair_r
\end{tikzcd}$$
where the lower left map is the natural open immersion identifying $\pair_r$ with the fiber over $\G_m/\G_m$ of $\DNC_r$. Then it exhibits $\pair_r$ as a retract of $\DNC_r$. Now take $\FF$ relative to the open parts. Since the central vertical map becomes an isomorphism after this operation, the diagram exhibits $\FF(\pair_r,\open_r)$ as a retract of $\FF(\E_r,\E_r\setminus0_r)$, i.e.\ $\beta_r^*$ has a left inverse. 

Note that the above diagram gives an endomorphism of $\FF(\E_r,\E_r\setminus0_r)$ by splicing $\beta_r\colon\E_r\to\pair_r$ to its lower left corner. We will finish the proof by showing that it is the identity, since then the left inverse of $\beta_r^*$ will also be the right inverse. To this end, we can base change the above diagram from $\pair_r$ to $\E_r$ and get the diagram
$$\begin{tikzcd}
    &\DNC_{r,0}\times_{\pair_r}\E_r\ar[r]\ar[d]&\E_r\\
    \E_r\ar[r]&\DNC_r\times_{\pair_r}\E_r\ar[ur,dashed]
\end{tikzcd}$$
where the upper right map is not the projection but the classifying map of the weighted normal bundle. Now by Lemma \ref{dncvb} below, there is a map indicated by the dashed arrow that makes the diagram commutes and exhibits $\E_r$ as a retract of $\DNC_r\times_{\pair_r}\E_r$, therefore completing the proof. 
\end{proof}

\begin{lemma}\label{dncvb}
The deformation to the normal cone of the zero section of a vector bundle $V$ on a stack $X$ is the zero section of the vector bundle $V(-1)$ on $X\times\A^1/\G_m$, where $(-1)$ denotes the twist by the inverse of the universal bundle on $\B\G_m$. Therefore, it restricts to $V$ itself on the open fiber, and restricts to the weighted normal bundle on the closed fiber. 
\end{lemma}

\begin{proof}
Recall that the deformation to the normal cone of $\Z[x]\to\Z$ is $\Z[xt^{-1},t]\to\Z[t]$. By its compatibility of base change and products, we know that the deformation to the normal cone of $B[x_1,\ldots,x_r]\to B$ is $B[x_1t^{-1},\ldots,x_rt^{-1},t]\to B[t]$, which is the zero section of the $(-1)$-twisted trivial bundle on $\spec(B)\times\A^1/\G_m$ as described. Note that this identification is compatible with change of coordinates, so we can globalize it and conclude. 
\end{proof}

\begin{remark}\label{relcharpair}
The proof of Theorem \ref{charclasspair} actually proves a stronger statement: Let $\FF$ be a big \'etale sheaf over $\pair_r$ that satisfies the same assumption with points taken over $\pair_r$. Then $\beta_r$ induces an isomorphism
$$\beta_r^*\colon\FF(\pair_r,\open_r)\cong\FF(\E_r,\E_r\setminus0_r),$$
still with points taken over $\pair_r$. 
\end{remark}

\begin{remark}\label{prodpair}
Let $\FF$ be as in Theorem \ref{charclasspair}, and $Y\to X$ be a map of stacks. Then $\GG(-)=\FF(X\times-,Y\times-)$ also satisfies the same assumption as $\FF$ does, since deformation to the normal cone commutes with $X\times-$ and $Y\times-$. Therefore
$$\FF((X,Y)\times(\pair_r,\open_r))\cong\FF((X,Y)\times(\E_r,\E_r\setminus0)),$$
where the product is understood as
$$(X_0,Y_0)\times(X_1,Y_1)=(X_0\times X_1,(Y_0\times X_1)\sqcup_{Y_0\times Y_1}(X_0\times Y_1)).$$
\end{remark}

In the rest of this subsection, let $\FF$ be as in Situation \ref{chernsit} and suppose $\FF$ satisfies the assumption of Theorem \ref{charclasspair} for all $r\in\N$. By the previous subsection, this holds for $\FF^n(R)$ being $\LO^n_R[n]$, $\RG_\et(R[1/p],\Z_p(n))[2n]$, and $\RG_\syn(R,\Z_p(n))[2n]$. 

\begin{definition}[Cycle classes]\label{defcyc}
For $r\in\N$, let $\cl_r^\FF\in\Omega^\infty\FF^r(\pair_r,\open_r)$ be the image of the universal Thom class $\Th_r^\FF\in\Omega^\infty\FF^r(\E_r,\E_r\setminus0_r)$ under the isomorphism in Theorem \ref{charclasspair}, called the \emph{universal cycle class} of codimension $r$. For a regular immersion $Y\to X$ of codimension $r$, let $\cl_{Y/X}^\FF\in\Omega^\infty\FF^r(X,X\setminus Y)$ be the pullback of $\cl_r^\FF$ along the classifying map $X\to\pair_r$, called the \emph{cycle class} of $Y$ in $X$. By abuse of notation, we also use $\cl_{Y/X}^\FF$ to denote its image in $\Omega^\infty\FF^r(X)$. 
\end{definition}

\begin{remark}[Relation with blowup]
Suppose the blowup formula holds, i.e.
$$\begin{tikzcd}
    \FF(D)&\FF(\tilde{X})\ar[l]\\
    \FF(Y)\ar[u]&\FF(X)\ar[l]\ar[u]
\end{tikzcd}$$
is a pullback square for every regular immersion $Y\to X$ of codimension $r$ with blowup $\tilde{X}$ and exceptional divisor $D$ (by \cite[Theorem 9.4.1]{apc} this is the case for syntomic cohomology). Suppose further that we have a decomposition
$$\FF^\star(\tilde{X})=\FF^\star(X)\oplus\bigoplus_{i=1}^{r-1}\FF^{\star-i}(Y)$$
functorial on $Y\to X$, extending the projective bundle formula of $D=\P(\NN_{Y/X})$. Denoting $U=X\setminus Y=\tilde{X}\setminus D$, one gets from above a similar decomposition
$$\FF^\star(\tilde{X},U)=\FF^\star(X,U)\oplus\bigoplus_{i=1}^{r-1}\FF^{\star-i}(Y).$$
In this case, $\cl_{Y/X}^\FF$ can also be defined as the image of $-c_1(-D)^r\in\Omega^\infty\FF^r(\tilde{X},U)$ under the projection to $\Omega^\infty\FF^r(X,U)$ (cf.\ \cite[Definition 1.1.2]{gabberpurity}). In fact, by functoriality and Theorem \ref{charclasspair}, it suffices to check the case $0_r\to\E_r$, which can be easily done by hand. 
\end{remark}

\begin{remark}[Uniqueness]\label{unicyc}
By construction, it is easy to see that the cycle class is uniquely determined by the following requirements:
\begin{description}
    \item[Functoriality] For $f\colon X'\to X$ a morphism of stacks with $Y'=Y\times_XX'$, we have $\cl_{Y'/X'}^\FF=f^*\cl_{Y/X}^\FF$. 
    \item[Normalization] For the zero section of a vector bundle, it is the Thom class in Theorem \ref{thom}. 
\end{description}
Therefore, the cycle class is natural in $\FF$ with respect to maps that preserve $c_1$, cf.\ Remark \ref{natthom}. 
\end{remark}

\begin{remark}[Multiplicativity]\label{mult}
By Theorem \ref{charclasspair} and the additivity of Thom classes, one easily deduces that $\cl_{r'}^\FF\boxtimes\cl_{r''}^\FF=m^*\cl_r^\FF$, where $r=r'+r''$ and $m\colon\pair_{r'}\times\pair_{r''}\to\pair_r$ represents the intersection. By Remark \ref{unithomadd} and Theorem \ref{charclasspair}, the cycle classes are uniquely determined by multiplicativity and the normalization in codimension $1$, which states that $\cl_{D/X}=c_1(\O(D))$ for a Cartier divisor $D\to X$. 
\end{remark}

The multiplicativity in the above remark has no homotopy coherence a priori, but the coconnectivity condition on $\FF$ actually forces the coherence. This will not be used in the sequel. 

\begin{lemma}\label{1hom}
Let $\CC^\otimes$ be a symmetric monoidal category and let $A,B\in\CAlg(\CC)$. Suppose for every $n\in\N$ that $\Hom_\CC(A^{\otimes n},B)$ is $0$-truncated. Then
$$\Hom_{\CAlg(\CC)}(A,B)=\Hom_{\CAlg(\h\CC)}(A,B)$$
and is $0$-truncated, where $\h\CC$ means the homotopy $1$-category of $\CC$. 
\end{lemma}

\begin{proof}
Recall from \cite[Definition 2.1.3.1]{ha} that $\CAlg(\CC)=\Fun_{\Fin_*}(\Fin_*,\CC^\otimes)$, and $A,B\in\CAlg(\CC)$ are actually functors $A^\otimes,B^\otimes\colon\Fin_*\to\CC^\otimes$ that takes $\<1\>$ to $A,B\in\CC$. As $\Hom$ in the functor category, $\Hom_{\CAlg(\CC)}(A,B)$ is built up as a limit of $\Hom_{\CC^\otimes}^f(A^\otimes_{\<n\>},B^\otimes_{\<m\>})$ for $n,m\in\N$ and $f\colon\<n\>\to\<m\>$ in $\Fin_*$. Since by definition $\CC^\otimes\to\Fin_*$ is a cocartesian fibration, we have
$$\Hom_{\CC^\otimes}^f(A^\otimes_{\<n\>},B^\otimes_{\<m\>})=\Hom_{\CC^\otimes_{\<m\>}}(\otimes_f(A^\otimes_{\<n\>}),B^\otimes_{\<m\>})$$
in the notation of \cite[Remark 2.1.2.16]{ha}; then by \cite[Remark 2.1.2.19]{ha} one easily computes that
$$\Hom_{\CC^\otimes_{\<m\>}}(\otimes_f(A^\otimes_{\<n\>}),B^\otimes_{\<m\>})=\prod_{i=1}^m\Hom_\CC(A^{\otimes f^{-1}(i)},B),$$
which is $0$-truncated by assumption. Therefore, as a limit of $0$-truncated animas, $\Hom_{\CAlg(\CC)}(A,B)$ is also $0$-truncated. Finally we conclude by noting that $\Hom_{\CAlg(\h\CC)}(A,B)$ is computed by exactly the same limit. 
\end{proof}

Let $\pair_\reg\subseteq\pair$ denote the stack of regular pairs without a fixed dimension; then it inherits the $\mathbb{E}_\infty$-monoid structure of $\pair$ with respect to the smash product, as described at the end of Definition \ref{pair}. Let $\Sigma^\infty$ denote the left adjoint of the functor $\Omega^\infty$ from the category of $\Sp$-valued big \'etale sheaves to the category of pointed stacks, by abuse of notation. Then since $\pair_\reg=(\pair_0)_+\vee\bigvee_{r\in\Z_+}\pair_r$ is graded, $\Sigma^\infty\pair_\reg$ is naturally a graded $\mathbb{E}_\infty$-ring. 

\begin{proposition}
The multiplicativity in Remark \ref{mult} uniquely determines a map $\Sigma^\infty\pair_\reg\to\FF$ of big \'etale sheaves with values graded $\mathbb{E}_\infty$-rings. 
\end{proposition}

\begin{proof}
By Lemma \ref{1hom}, it suffices to check for every $n\in\N$ that 
$$\Hom_{\Fun(\Ring,\Gr\Sp)}((\Sigma^\infty\pair_\reg)^{\otimes n},\FF)\in\Ani_{\le0}.$$
Passing to each grade and using the fact that $\Sigma^\infty$ takes the smash product to the tensor product, we need to check that for $r_1,\ldots,r_n\in\Z_+$ and $r=\sum_{i=1}^nr_i$,
$$\tilde\FF^r\left(\bigwedge_{i=1}^n\pair_{r_i}\right)\in\Sp_{\le0}.$$
Now by Remark \ref{prodpair}, 
$$\tilde\FF\left(\bigwedge_{i=1}^n\pair_{r_i}\right)=\FF\left(\prod_{i=1}^n(\E_{r_i},\E_{r_i}\setminus0)\right),$$
so the desired coconnectivity follows by applying Proposition \ref{cr0} $n$ times. 
\end{proof}

Cycle classes can be upgraded to Gysin maps as follows: 

\begin{definition}[Gysin maps]\label{cyclemap}
Fix $r\in\N$. Let $Y\to X$ be a regular immersion of codimension $r$, $\Y\to\X$ be its deformation to the normal cone and $\Y_0\to\X_0$ be the closed fiber. Then the \emph{Gysin map} of $Y\to X$ is constructed as the composition
$$\FF^\star(Y)\to\FF^\star(\Y_0)\to\FF^{\star+r}(\X_0,\X_0\setminus\Y_0)=\FF^{\star+r}(\X,\X\setminus\Y)\to\FF^{\star+r}(X,X\setminus Y),$$
where the first arrow comes from the projection $\Y_0=Y\times0/\G_m\to Y$, the second arrow is the Thom map, and the final arrow is restriction to the open fiber. By Theorem \ref{charclasspair}, $\cl^\FF_{Y/X}$ is the image of $1\in\Omega^\infty\FF^0(Y)$ under the Gysin map. By the definition of the Thom map, after composing with the restriction $\FF(X,X\setminus Y)\to\FF(Y)$, the Gysin map becomes multiplying $c_r(\NN_{Y/X})$. 
\end{definition}

We now identify the Gysin map for the Hodge cohomology. For a regular pair $R\to S$, the functor $i_*\colon\D(S)\to\D(R)$ commutes with colimits, preserves compact objects, and hence has a right adjoint $i^!\colon\D(R)\to\D(S)$ that preserves colimits. 

\begin{lemma}
$i^!\colon\D(R)\to\D(S)$ satisfies the projection formula $i^!-=i^!R\otimes_Si^*-$, and is compatible with base change of pairs. 
\end{lemma}

\begin{proof}
For $M\in\D(R)$, the natural map
$$M=\inhom_R(R,M)\to\inhom_R(i_*i^!R,M)=i_*\inhom_S(i^!R,i^!M)$$
gives by adjunction the desired map
$$i^!R\otimes_Si^*M\to i^!M.$$
Now it is an isomorphism when $M=R$, and both sides commute with colimits in $M$, so it is an isomorphism for all $M\in\D(R)$. For the compatibility with base change, one can check after $i_*$, which follows from the fact that $i_*i^!-=\inhom_R(S,-)=S^\vee\otimes_R-$, since by assumption $S$ is dualizable in $\D(R)$.
\end{proof}

Now consider a regular immersion $i\colon Y\to X$ of stacks of codimension $r$. Recall by Remark \ref{qcoh} that we still have the functor $i_*\colon\D(Y)\to\D(X)$. Since colimits of quasicoherent sheaves can be computed locally, $i_*$ still commutes with colimits, is compatible with base change, and thus has a right adjoint $i^!$ that also has these properties and satisfies the projection formula. 

\begin{lemma}\label{uppershriek}
There is a natural identification of $i^!\O_X=\det(\NN_{Y/X})[-r]$ compatible with base change and product of pairs, where $\NN_{Y/X}$ denotes the normal bundle. 
\end{lemma}

\begin{proof}
Computing locally, we see that $i^!\O_X$ is an invertible object in $\D(Y)$. To identify it, we first consider the case where $\pi\colon X\to Y$ is a rank $r$ vector bundle and $i\colon Y\to X$ is the immersion of the zero section. Let $\O=\O_Y$, $\NN=\NN_{Y/X}$, and $\AA=\lsym_\O^\star(\NN^\vee)=\pi_*\O_X$. Now the Koszul sequence
$$0\to\AA\otimes\Lambda^r\NN^\vee\to\cdots\to\AA\otimes\NN^\vee\to\AA\to\O\to0$$
gives a natural identification $\inhom_\AA(\O,\AA)=\det(\NN)[-r]$ as an $\AA$-module, so
$$i^!\O_X=\pi_*i_*i^!\O_X=\pi_*\inhom_\AA(\O,\AA)=\det(\NN)[-r]$$
as an $\O$-module. 

For a general $i\colon Y\to X$, we consider its deformation to the normal cone $\ii\colon\Y\to\X$, whose closed fiber $\ii_0\colon\Y_0\to\X_0$ is naturally the vector bundle $\NN_{Y/X}/\G_m$ on $\Y_0=Y\times\B\G_m$. Therefore by the previous paragraph, $\ii^!\O_\X$ is an invertible object in $\D(\Y)$, whose restriction on $\Y_0=Y\times\B\G_m$ is $\det(\NN_{Y/X})[-r]$ with $\G_m$ acting by weight $r$. But $\Y=Y\times\A^1/\G_m$, so $\D(\Y)$ can be viewed as the category of graded quasicoherent $\O_Y[t]$-modules. Now $\ii^!\O_\X$ is an invertible graded $\O_Y[t]$-modules which gives $\det(\NN_{Y/X})[-r]$ on degree $r$ after quotient by $t$. Hence
$$\ii^!\O_\X=\bigoplus_{n=r}^\infty\det(\NN_{Y/X})t^n[-r],$$
and obviously its restriction to the open part $Y\to\Y$ is $\det(\NN_{Y/X})$. 

The compatibility with base change and product follows from the compatibility of both the deformation to the normal cone and the Koszul sequence. 
\end{proof}

\begin{construction}
Recall from the shifted fiber sequence of cotangent complexes
$$\L_{Y/X}[-1]\to i^*\L_X\to\L_Y$$
that the normal bundle $\NN_{Y/X}$ can be identified as $\L_{Y/X}[-1]^\vee$. Since $\L_{Y/X}[-1]$ is locally free of rank $r$, the above sequence gives a natural map
$$\Lambda^r(\L_{Y/X}[-1])\otimes\LO^n_Y\to i^*\LO^{n+r}_X;$$
since $\Lambda^r(\L_{Y/X}[-1])=(i^!\O_X[r])^\vee$ by Lemma \ref{uppershriek}, we can take it to the right hand side, use the projection formula, and obtain a Gysin map
$$\LO^n_Y\to i^!\O_X\otimes\LO^{n+r}_X[r]=i^!\LO^{n+r}_X[r].$$
\end{construction}

\begin{proposition}\label{hodgecycle}
The Gysin map for the Hodge cohomology is the global section of the adjunction $i_*\LO^n_Y\to\LO^{n+r}_X[r]$ of the above map. Since $i_*\LO^n_Y$ restricts to $0$ on $X\setminus Y$, it actually maps to the relative Hodge cohomology of $(X,X\setminus Y)$. 
\end{proposition}

\begin{proof}
Consider the deformation to the normal cone $\Y\to\X$ and its closed fiber $\Y_0\to\X_0$. Since the immersion of the open part exhibits the pair $Y\to X$ as a retract of $\Y\to\X$, it suffices to prove the proposition for $\Y\to\X$. Now by Corollary \ref{0hg} the relative Hodge cohomology of $(\X,\X\setminus\Y)$ coincides with that of $(\X_0,\X_0\setminus\Y_0)$, so it suffices to prove the same for $\Y_0\to\X_0$. 

Therefore, we can assume that $\pi\colon X\to Y$ is a rank $r$ vector bundle and $i\colon Y\to X$ is the immersion of the zero section, and try to prove that the map described above is the Thom map. In this case, the Hodge cohomology of $X$ is naturally a module of the Hodge cohomology of $Y$ by $\pi^*$, and the map described is a module map, so it suffices to prove that the image of $1$ is $\Th_{X/Y}$. By Remark \ref{unithomadd}, this reduces to checking the functoriality, the additivity, and the normalization for line bundles for this image. Now both the functoriality and the additivity follow from Lemma \ref{uppershriek}. For the normalization for line bundles, we work in the universal case $\B\G_m\to\A^1/\G_m$ by identifying quasicoherent sheaves on $\A^1/\G_m$ as graded $\Z[t]$-modules, where it is routine to check that the map $i_*\O_{\B\G_m}\to\L_{\A^1/\G_m}[1]$ is $1$ under the identification of Example \ref{a1h}, which coincides with $\Th_1^\hod$ by Remark \ref{c1hod}. 
\end{proof}

Similar method can be used to enhance the Gysin map of the $p$-completed topological cyclic homology to a motivic filtered map. First we recall the filtration: 

\begin{proposition}\label{motfiltc}
There is a natural exhaustive filtration $\fil^\bullet\TC(R)_p^\wedge$ on the $p$-completed $\TC$, functorial in $R\in\Ring$, such that:
\begin{enumerate}
    \item\label{grtc} $\gr^i\TC(R)_p^\wedge=\Z_p(i)(\hat{R})[2i]$. In particular, $\fil^0\TC(R)_p^\wedge=\TC(R)_p^\wedge$. 
    \item\label{filtcsifted} The functor $\fil^\bullet\TC_p^\wedge\colon\Ring\to\Fil(\Sp_p^\wedge)$ commutes with sifted colimits, and hence is left Kan extended from polynomial rings of finite type over $\Z$. 
    \item\label{filtcetsh} $\fil^\bullet\TC_p^\wedge$ is a big \'etale sheaf, and is even a $p$-quasisyntomic sheaf when restricted to classical $p$-bounded $p$-quasisyntomic rings. 
\end{enumerate}
\end{proposition}

\begin{proof}
Note that $\TC(R)_p^\wedge$ only depends on $\hat{R}$, since $\THH(R)_p^\wedge$ is so. Also by \cite[Corollary 2.15]{cmm}, the functor $\TC_p^\wedge\colon\Ring\to\Sp_p^\wedge$ commutes with sifted colimits, so we can construct the filtration by left Kan extension from classical $p$-bounded $p$-quasisyntomic rings, where it follows from \cite[\S 7.4]{bms2}, which automatically implies (\ref{filtcetsh}). Now (\ref{grtc}) and (\ref{filtcsifted}) follow from \cite[Proposition 7.4.8]{apc}. 
\end{proof}

Recall that $\TC(R)$ is actually an invariant of the category $\D(R)^\omega$ of perfect $R$-complexes, since $\THH$ is so, cf.\ \cite[\S 10.1]{univkt}. Hence if $R\to S$ is a regular pair, then there is a wrong-way map $\D(S)^\omega\to\D(R)^\omega$, since perfect $S$-complexes are perfect $R$-complexes. This induces a pushforward map $\TC(S)\to\TC(R)$, which can be globalized to a pushforward map $\TC(Y)\to\TC(X)$ associated to every regular immersion $Y\to X$ of stacks. Here the meaning of $\TC(X)$ is as in Remark \ref{valsta}. Since the composition $\D(S)^\omega\to\D(R)^\omega\to\D(\spec(R)\setminus\spec(S))^\omega$ is functorially zero, the map $\TC(Y)\to\TC(X)$ factors through $\TC(X,X\setminus Y)$. 

\begin{theorem}[The Gysin map for $\fil^\bullet\TC_p^\wedge$]\label{cyctc}
Fix $r\in\N$. There is a unique way (up to contractible choice) to enhance the map $\TC(Y)_p^\wedge\to\TC(X,X\setminus Y)_p^\wedge$ to a filtered map $\fil^\bullet\TC(Y)_p^\wedge\to\fil^{\bullet+r}\TC(X,X\setminus Y)_p^\wedge$ functorially for all regular immersions $Y\to X$ of codimension $r$. Moreover, its associated graded is the Gysin map for the syntomic cohomology of $p$-formal schemes as in Definition \ref{cyclemap}. 
\end{theorem}

\begin{proof}
For notational simplicity, we omit all the $-_p^\wedge$, and use $\Z_p(i)$ to denote the functor $R\mapsto\Z_p(i)(\hat{R})$. It suffices to prove that for integers $n\ge n'$, the map
\begin{equation}\label{compo}
    \fil^n\TC(Y)\to\TC(Y)\to\TC(X,X\setminus Y)\to\TC(X,X\setminus Y)/\fil^{n'+r}\TC(X,X\setminus Y)
\end{equation}
has a unique functorial nullhomotopy. As before, let $\Y\to\X$ denote the deformation to the normal cone of $Y\to X$, and let $\Y_0\to\X_0$ denote its closed fiber. Note that the above map can be written as
\begin{align*}
    &\quad\ \!\fil^n\TC(Y)\to\fil^n\TC(\Y_0)\to\TC(\X_0,\X_0\setminus\Y_0)/\fil^{n'+r}\TC(\X_0,\X_0\setminus\Y_0)\\
    &\cong\TC(\X,\X\setminus\Y)/\fil^{n'+r}\TC(\X,\X\setminus\Y)\to\TC(X,X\setminus Y)/\fil^{n'+r}\TC(X,X\setminus Y),
\end{align*}
where the first arrow comes from the projection $\Y_0\to Y$, the second arrow is the map (\ref{compo}) for $\Y_0\to\X_0$, the isomorphism is due to Theorem \ref{syndnc} since $\TC/\fil^{n'+r}\TC$ is finitely filtered by $\Z_p(i)[2i]$, and the final arrow is restriction to the open fiber. Recall that $\X_0=\NN_{Y/X}/\G_m$ is the weighted normal bundle over $\Y_0=Y\times\B\G_m$. Therefore the theorem is reduced to the following statement: the map
$$\fil^n\TC(Y)\to\TC(V/\G_m,(V\setminus0)/\G_m)/\fil^{n'+r}\TC(V/\G_m,(V\setminus0)/\G_m)$$
has a unique nullhomotopy that is functorial in the rank $r$ vector bundle $V\to Y$. Now the right hand side is finitely filtered by
$$\Z_p(i+r)(V/\G_m,(V\setminus0)/\G_m)[2(i+r)]=\Z_p(i)(Y\times\B\G_m)[2i]$$
for $i<n'$ by the Thom isomorphism, so it suffices to show that for $i<n'$, any map
$$\fil^n\TC(Y)\to\Z_p(i)(Y\times\B\G_m)[2i]$$
that is functorial in $Y$ is uniquely nullhomotopic. By \'etale descent we can assume that $Y=\spec(R)$ is affine. Then by descent along the \v{C}ech nerve of $Y\to Y\times\B\G_m$, we reduce to show that, for $i<n'$ and any ring $A$ smooth over $\Z$, any map
$$\fil^n\TC(R)\to\Z_p(i)(R\otimes A)[2i]$$
that is functorial in $R$ is uniquely nullhomotopic. Since both sides commute with sifted colimits in $R$, we can assume that $R$ is also smooth over $\Z$. Now both sides satisfy quasisyntomic descent, and the theorem follows from the fact that quasisyntomic locally, $\fil^n\TC$ is $2n$-connective, but $\Z_p(i)[2i]$ lies in homotopical degree $2i<2n'\le2n$, cf.\ \cite[\S 7.4]{bms2} and \cite[Theorem 14.1]{bs19}. 

To see that the associated graded gives the Gysin map, by the same deformation to the normal cone argument, we only need to show that for a rank $r$ vector bundle $V\to Y$, it gives the syntomic Thom map
$$\Z_p(\star)(Y)[2\star]\to\Z_p(\star+r)(V,V\setminus0)[2(\star+r)]$$
when applied to the zero section $0\to V$. This is a map of $\Z_p(\star)(Y)[2\star]$-modules, so it suffices to show that the image of $1\in\Z_p(0)(Y)$ is $\Th_V^\syn\in\Z_p(r)(V,V\setminus0)[2r]$. By Remark \ref{unithomadd}, this reduces to the functoriality, the additivity, and the normalization for line bundles for this image, which we now check:
\begin{description}
    \item[Functoriality] This is clear by the functoriality of the Gysin map.
    \item[Additivity] By \cite[Theorem 1.7]{cmm}, the $p$-completed $\TC$ is the \'etale sheafification of the $p$-completed connective $K$-theory. Now $1\in K(Y)$ is just $[\O_Y]$, and the product is just the tensor product, so the additivity follows from the fact that if $V=V'\oplus V''$ is a direct sum, then $\O_0=\O_{V'}\otimes_{\O_V}\O_{V''}$, where we view $V'$ and $V''$ as closed substacks of $V$. 
    \item[Normalization for line bundles] As above, viewing $\TC$ as the \'etale $K$-theory, it suffices to show that for a Cartier divisor $D\to X$, the associated graded class of $[\O_D]\in\fil^1\TC(X)$ is $c_1(\O(D))\in\Omega^{\infty-2}\Z_p(1)(X)$. By \'etale descent, we can assume that $X=\spec(R)$ is affine. Then by \cite[Proposition 7.4.8]{apc}, we can assume that $R$ is smooth over $\Z$. Now by $p$-quasisyntomic descent, we can assume that $R$ is over $\Z_p^\cyc$ and is quasiregular semiperfectoid, where it follows from \cite[Theorem 6.7]{cycdtr} and \cite[Corollary 2.6.11, Proposition 7.5.2]{apc}. \qedhere
\end{description}
\end{proof}

\begin{remark}
The $p$-completion $K^\sel(R)_p^\wedge$ of the Selmer $K$-theory introduced in \cite[\S 2]{kartin} is an enhancement of $\TC(R)_p^\wedge$ that contains information about $R[1/p]$, cf.\ \cite[\S 6]{hyperetk}, and it should be motivic filtered by the non-$p$-formal syntomic cohomology sheaf $\Z_p(i)(R)[2i]$. We expect an analog of Theorem \ref{cyctc} in this setting. 
\end{remark}

\section{Poincar\'e duality}

We finally prove the prismatic Poincar\'e duality in this section. First we recall its Hodge version in \cite[Lecture 16]{dr}. Let $S$ be a stack and $X$ be a stack over $S$. Suppose the structure map $f\colon X\to S$ is representable by algebraic spaces, and is proper smooth of dimension $d\in\N$. Then the relative diagonal $\Delta\colon X\to X\times_SX$ is a regular immersion of codimension $d$. Note by \cite[\texttt{0A1P}]{stacks} that the relative Hodge cohomology $f_*(\LO_{X/S}^\star)[-\star]$ is dualizable in $\D(S)$. 

\begin{theorem}[Poincar\'e duality for Hodge cohomology]\label{hodgedual}
The image of the absolute Hodge cycle class $\cl^\hod_\Delta\in\Omega^{\infty-d}\RG(X\times_SX,\LO_{X\times_SX}^d)$ along
\begin{align*}
    &\quad\ \RG(X\times_SX,\LO_{X\times_SX}^\star)\to\RG(X\times_SX,\LO_{(X\times_SX)/S}^\star)\\
    &=\RG(S,(f\times_Sf)_*(\LO_{(X\times_SX)/S}^\star))=\RG(S,f_*(\LO_{X/S}^\star)\otimes f_*(\LO_{X/S}^\star))
\end{align*}
gives rise to a graded map
$$\O_S\to f_*(\LO_{X/S}^\star)\otimes f_*(\LO_{X/S}^\star)(d)[d],$$
where $(d)$ and $[d]$ denote the shift on the grade and the homological degree, respectively. This map is a perfect copairing, i.e.\ the dual map
$$p_*(\LO_{X/S}^\star)^\vee\to f_*(\LO_{X/S}^\star)(d)[d]$$
is an isomorphism. 
\end{theorem}

\begin{proof}
The proof is the same as in \cite[Lecture 16]{dr}, with our Proposition \ref{hodgecycle} in place of its Lemma 2. For convenience and completeness, we reproduce it here. 

Since the cycle class is the image of $1$ under the Gysin map, using Proposition \ref{hodgecycle} with $n=0$ we see that the copairing map in the theorem, in individual grades, can be described as follows: take \th{$d$} wedge power of the canonical map from the conormal to the cotangent
$$\LO_{X/S}^d=\Lambda^d(\L_{X/(X\times_SX)}[-1])\to\Delta^*\LO_{(X\times_SX)/S}^d;$$
note that $\Delta^!\O_{X\times_SX}=(\LO_{X/S}^d)^{-1}[-d]$; tensor this with both sides to get
$$\O_X[-d]\to\Delta^*\LO_{X\times_SX}^d\otimes\Delta^!\O_{X\times_SX}=\Delta^!\LO_{X\times_SX}^d;$$
then move the $\Delta^!$ to the left, apply $(f\times_Sf)_*$, and finally compose with the canonical map $\O_S\to f_*\O_X$ and the projection to one direct summand in K\"unneth to get
$$\O_S[-d]\to f_*\O_X[-d]\to(f\times_Sf)_*\LO_{X\times_SX}^d\to f_*\LO_{X/S}^i\otimes f_*\LO_{X/S}^j,$$
where $i$ and $j$ are fixed with $i+j=d$. To see that this is perfect, we apply Lemma \ref{clausen16.3} to reduce to perfectness of
$$\LO_{X/S}^d\to\Delta^*\LO_{X\times_SX}^d=\LL^d(\L_{X/S}\oplus\L_{X/S})\to\LO_{X/S}^i\otimes\LO_{X/S}^j,$$
where the final map is the projection to one direct summand; now this is a claim on quasicoherent sheaves over $X$, so it can be verified \'etale locally, where it follows by a straightforward calculation with local coordinates. 
\end{proof}

\begin{lemma}[{cf.\ \cite[Lemma 16.3]{dr}}]\label{clausen16.3}
    Suppose $M$ and $N$ are perfect complexes over $X$ that are in duality via a perfect copairing
    $$\LO_{X/S}^d[d]=f^!\O_S\to M\otimes N=\Delta^*(M\boxtimes_SN).$$
    Then $f_*M$ and $f_*N$ are in duality via the copairing
    $$\O_S\to f_*M\otimes f_*N$$
    given as follows: first tensor both sides by $\Delta^!\O_{X\times_SX}=(f^!\O_S)^{-1}$ to get
    $$\O_X\to\Delta^*(M\boxtimes_SN)\otimes\Delta^!\O_{X\times_SX}=\Delta^!(M\boxtimes_SN);$$
    then move the $\Delta^!$ to the left, apply $(f\times_Sf)_*$, and finally compose with the canonical map $\O_S\to f_*\O_X$ to get
    $$\O_S\to f_*\O_X\to(f\times_Sf)_*(M\boxtimes_SN)=f_*M\otimes f_*N.$$
\end{lemma}

\begin{proof}
    Take a pairing $N\otimes M\to f^!\O_S$ that forms a duality datum with the copairing. Move $f^!$ to the left and use the lax monoidal structure of $f^*$ to get
    $$f_*N\otimes f_*M\to f_*(N\otimes M)\to\O_S.$$
    It suffices to show that this pairing forms a duality datum with the copairing given above. To this end, we consider the composition
    $$f^!\O_S\otimes M\to(M\otimes N)\otimes M=M\otimes(N\otimes M)\to M\otimes f^!\O_S.$$
    Tensoring $\Delta^!\O_{X\times_SX}=(f^!\O_S)^{-1}$, we get
    $$M\to\Delta^!(M\boxtimes_S(N\otimes M))\to\Delta^!(M\boxtimes_Sf^!\O_S)=M.$$
    Moving $\Delta^!$ to the left and applying $(f\times_Sf)_*$, we get
    $$f_*M\to f_*M\otimes f_*(N\otimes M)\to f_*M\otimes f_*f^!\O_S\to f_*M,$$
    where the last map is the counit. Now since we start with a duality datum, the first composition is the map commuting $M$ and $f^!\O_S$, so the second and the third compositions are identity maps. Therefore, it suffices to see that the first map above is given by the composition
    $$f_*M\to(f_*M\otimes f_*N)\otimes f_*M=f_*M\otimes(f_*N\otimes f_*M)\to f_*M\otimes f_*(N\otimes M),$$
    where the first map is induced by the copairing we constructed. Since we are working with quasicoherent sheaves on $S$, it suffices to show the statement after restricted along every morphism $U\to S$ and composed with every map $\varphi\colon\O_U[n]\to f_*M|_U$ for every $n\in\Z$. Restriction commutes with everything here, so we can assume $U=S$. Then the statement follows from the functorialities of all the above constructions applied to the map $N[n]\to N\otimes M$ induced by $\varphi$. 
\end{proof}

Now let $(A,I)$ be a prism and take $S=\spf(\bar{A})$. To prove the Poincar\'e duality, we need to produce a copairing map using the syntomic cycle class. 

\begin{remark}\label{syntosect}
Let $R$ be a ring over $\bar{A}$. Note that the comparison map $\fil^\bullet\Prism_R\{d\}\to(\fil^\bullet\Prism_{R/A}^{(1)})\{d\}$ in \cite[Construction 5.6.1]{apc} is compatible with Frobenius, i.e.\ the diagram
$$\begin{tikzcd}
\fil^d\Prism_R\{d\}\ar[rr,"\varphi_R\{d\}"]\ar[d]&&\Prism_R\{d\}\ar[d]\\
(\fil^d\Prism_{R/A}^{(1)})\{d\}\ar[r,"\varphi_{R/A}\{d\}"']&\Prism_{R/A}\{d\}\ar[r,"-\otimes1"']&\Prism_{R/A}^{(1)}\{d\}
\end{tikzcd}$$
commutes. By Proposition \ref{glsect}, this gives rise to a comparison map
$$\Z_p(d)(R)\to\Hom_{\FGauge(A,I)}(A,\Prism_{R/A}\{d\}).$$
Also, it is not hard to see from \cite[Remark 5.5.8]{apc} that the diagram
$$\begin{tikzcd}
    \fil^d\Prism_R\{d\}\ar[r]\ar[d]&\widehat{\LO^d_R}[-d]\ar[d]\\
    (\fil^d\Prism_{R/A}^{(1)})\{d\}\ar[r]&\widehat{\LO^d_{R/\bar{A}}}[-d]
\end{tikzcd}$$
commutes. Therefore, for a regular immersion $Z\to Y$ of stacks over $\spf(\bar{A})$ of codimension $d$, the cycle class $\cl^\syn_{Y/X}\in\Omega^{\infty-2d}\Z_p(d)(X)$ gives rise to a map $A\to\RG_\Prism(X/A)\{d\}[2d]$ that sends $1$ to the Hodge cycle class after Hodge specialization. 
\end{remark}

Recall that $X$ is a stack over $S=\spf(\bar{A})$, and the map $X\to S$ is representable by algebraic spaces proper smooth of dimension $d$. Recall by Corollary \ref{cohdual} that $\RG_\Prism(X/A)$ is dualizable in $\FGauge(A,I)$. 

\begin{theorem}[Prismatic Poincar\'e duality]\label{prismdual}
By Remark \ref{syntosect}, the syntomic cycle class $\cl^\syn_\Delta$ of the relative diagonal $\Delta\colon X\to X\times_{\spf(\bar{A})}X$ induces a map of $F$-gauges
$$A\to\RG_\Prism((X\times_{\spf(\bar{A})}X)/A)\{d\}[2d]=\RG_\Prism(X/A)\otimes_A\RG_\Prism(X/A)\{d\}[2d].$$
This map is a perfect copairing, i.e.\ the dual map
$$\RG_\Prism(X/A)^\vee\to\RG_\Prism(X/A)\{d\}[2d]$$
is an isomorphism. 
\end{theorem}

\begin{proof}
By Proposition \ref{cons}, it suffices to prove that both the Hodge--Tate and the de Rham specializations of the above map are isomorphisms. By Proposition \ref{dual}, both specializations are complete with only finitely many nonzero graded pieces, so we can pass to the graded pieces, which are the Hodge cohomologies, and the theorem follows from Theorem \ref{hodgedual} and Remark \ref{syntosect}. 
\end{proof}

\begin{remark}[Trace maps]\label{prismtrace}
Theorem \ref{prismdual} gives a trace map $\RG_\Prism(X/A)\{d\}[2d]\to A$ as the dual of the structure map $A\to\RG_\Prism(X/A)$, which reduces to the Hodge trace map after specialization. Now if there were a suitably defined dualizable coefficient system $\EE$ on $X$, then the pairing $\EE\otimes\EE^\vee\to\O$ along with the trace map gives a canonical pairing
$$\RG_\Prism(X/A,\EE)\otimes\RG_\Prism(X/A,\EE^\vee)\{d\}[2d]\to\RG_\Prism(X/A)\{d\}[2d]\to A.$$
If one worked out the theory of relative $F$-gauges as sketched in Remark \ref{relcoef}, then it would be not hard to prove that the pairing is perfect by Hodge specialization; alternatively, one could also see that it is already perfect for $\EE$ being a dualizable prismatic crystal by Hodge--Tate specialization, via a careful study of the cohomology along $\wcart_{X/A}^\HT\to X$, cf.\ \cite[Proposition 5.12, Corollary 6.6]{pst}. 
\end{remark}

\begin{remark}[Comparisons]
The comparisons of Theorem \ref{prismdual} with the duality maps of several other cohomology theories should follow from the uniqueness of cycle classes in Remark \ref{unicyc} and Remark \ref{mult}. More specifically:
\begin{itemize}
    \item By the proof of Theorem \ref{prismdual}, it is already clear that the prismatic duality map specializes to the Hodge duality map given by the coherent duality and the de Rham duality map constructed in \cite[Lecture 16]{dr}. 
    \item For $I=pA$, Theorem \ref{prismdual} recovers the crystalline Poincar\'e duality in \cite[Theorem VII.2.1.3]{bertcrys} via \cite[Theorem 5.2]{bs19}. Indeed, under the K\"unneth formula, Berthelot's pairing map is the composition of the diagonal pullback with the trace map, so by \cite[Corollary VII.2.3.2]{bertcrys} its dual copairing is the cycle class of the diagonal, which is exactly the copairing here. 
    \item If one started with a perfect prism $(A,I)$ and a smooth and proper algebraic space over $\spec(\bar{A})$, rather than such a formal algebraic space over $\spf(\bar{A})$, then the syntomic cycle class would specialize to the \'etale one after inverting $p$, so the prismatic duality map would also specialize to the \'etale one after inverting $p$. We are not going into details here, since in future projects, we plan to compare the prismatic Poincar\'e duality with that of diamonds in \cite[Theorem 3.10.20]{mann}. 
\end{itemize}
\end{remark}

\begin{remark}
In future projects, we hope to explore generalizations of Theorem \ref{prismdual} in the following directions: 
\begin{itemize}
    \item to general coefficients for the prismatic cohomology as in Remark \ref{prismtrace}.
    \item to the absolute prismatic cohomology, or more generally to the coherent duality along the map $X^{\Prism''}\to S^{\Prism''}$ of prismatizations as in \cite{drinprism}.
    \item to non-proper maps, i.e.\ to a 6-functor formalism with a condensed lower shriek as in \cite{cond}. 
    \item to non-$p$-formal schemes, i.e.\ gluing with similar structure in the \'etale cohomology, as in \cite[\S 8.4]{apc}. 
\end{itemize}
\end{remark}

\printbibliography

\end{document}